\documentclass[10pt,twocolumn]{article}

\RequirePackage{amsthm,amsmath,amsfonts,amssymb}
\RequirePackage[authoryear]{natbib}
\RequirePackage[colorlinks,linkcolor=blue,citecolor=blue,urlcolor=blue]{hyperref}
\RequirePackage{graphicx}
\usepackage{fix-cm}
\usepackage{xcolor}
\usepackage{appendix}
\usepackage{geometry}
\usepackage{enumitem}

\theoremstyle{plain}
\newtheorem{theorem}{Theorem}[section]

\newtheorem{proposition}[theorem]{Proposition}

\theoremstyle{remark}
\newtheorem{definition}[theorem]{Definition}
\newtheorem{assumption}[theorem]{Assumption}
\newtheorem{example}[theorem]{Example}

\def\Real{\mathbb{R}}
\def\A{{\mathcal A}}
\def\H{{\mathcal H}}

\newcommand{\dd}{\mathrm{d}}

\DeclareMathOperator{\sech}{sech}
\DeclareMathOperator{\TV}{TV}

\DeclareMathOperator{\lfdr}{lfdr}
\DeclareMathOperator{\lnsr}{lnsr}
\DeclareMathOperator{\lidr}{clar}  
\DeclareMathOperator{\clar}{clar}  
\DeclareMathOperator{\lfsr}{lfsr}

\DeclareMathOperator{\Var}{Var}

\def\V{{\mathcal V}}

\def\indep{\mathrel{\rlap{${\perp}$}\kern 2pt{\perp}}}
\def\given{\mid}
\def\q{\alpha}
\def\tol{\delta}
\def\de{d}
\def\R{\mathbb{R}}
\def\E{\textnormal{E}}
\def\P{\mathrm{P}}
\def\H{\mathrm{H}}

\def\e{\varepsilon}

\usepackage{tikz}
\usetikzlibrary{arrows.meta, positioning, calc}

\title{Interpretation of local false discovery rates under the zero assumption}
\author{Daniel Xiang\\
\texttt{dxiang@uchicago.edu}\\
\and 
Nikolaos Ignatiadis\\
\texttt{ignat@uchicago.edu}\\
\and
Peter McCullagh\\
\texttt{pmcc@galton.uchicago.edu} \\
}
\date{March, 2025}

\begin{document}
\twocolumn[
  \begin{@twocolumnfalse}
    \maketitle
    \begin{abstract}
    In large-scale studies with parallel signal-plus-noise observations, the local false discovery rate is a summary statistic that is often presumed to be equal to the posterior probability that the signal is null. We prefer to call the latter quantity the local null-signal rate to emphasize our view that a null signal and a false discovery are not identical events. The local null-signal rate is commonly estimated through empirical Bayes procedures that build on the `zero density assumption,' which attributes the density of observations near zero entirely to null signals. In this paper, we argue that this strategy does not furnish estimates of the local null-signal rate, but instead of a quantity we call the complementary local activity rate (clar). Although it is likely to be small, an inactive signal is not necessarily zero. The clar dominates both the local null-signal rate and the local false sign rate and is a weakly continuous functional of the signal distribution. As a consequence, it takes on sensible values when the signal is sparse but not exactly zero. Our findings clarify the interpretation of local false discovery rates estimated under the zero density assumption.    
    \vspace{1cm}
    \end{abstract}

  \end{@twocolumnfalse}
]

\section{Introduction}

A great part of the literature on signal detection begins with the scalar signal-plus-Gaussian noise model~\citep{johnstone2004needles}, 
\begin{equation}
\label{eq:structural}
 X \sim \P, \quad\ Y \mid X \sim \mathrm{N}(X, 1),
\end{equation}
where $X$ is the unobserved signal and $Y$~is the observation. The concept of \lq false discovery\rq\ and the control of false discovery rates
plays a fundamental role in a wide range of procedures and algorithms
\citep{benjamini1995controlling, efron2010largescale, barber2015controlling}. 
Most discussions of local false discovery begin with the local null-signal rate
\begin{equation}
    \lnsr(y) := \P(X=0 \mid Y=y),
\label{eq:localnsr}
\end{equation}
where the event $X=0$ is the presumed inferential target. 

In modern large-scale inference problems, it is common to observe data $Y_1,\dotsc,Y_n$ drawn independently from a large number of distinct units or sites according to~\eqref{eq:structural}, so the empirical Bayes principle~\citep{robbins1956empirical, efron2010largescale} permits (and encourages) estimation of the signal distribution and the inferential target based on the data. 

The most widely used estimates start with a structure-agnostic two-groups model~\citep{efron2001empirical}.
Under \eqref{eq:structural}, the marginal density of $Y$ is a convex combination of the `null' Gaussian density $f_0 = \phi$ and a `non-null' density $f_1$,
\begin{equation}
f_{\eta_1}(y) =(1-\eta_1) \phi(y)+\eta_1 f_1(y),
\label{eq:generic_two_groups_model}
\end{equation}
with $\eta_1 \in [0,1]$.
The two-groups interpretation is a joint distribution on $\Real\times\{0,1\}$ such that each site is `null' with probability $1-\eta_1$, or `non-null' with probability $\eta_1$. The observation $Y$ for a `null' signal is drawn from the standard normal density, while for a `non-null' signal it is drawn from the density $f_1$ that is left unspecified. The statistician does not observe the `null' or `non-null' status, and so only observes draws from the marginal density $f_{\eta_1}$ in~\eqref{eq:generic_two_groups_model}; see Fig.~\ref{fig:two-groups-diagram} for a schematic. 

Empirical Bayes estimation strategies in the two-groups model often build on the `zero density assumption'~\citep{efron2010largescale} which attributes the density of observations near zero entirely to `null' signals. More formally:

\begin{assumption}[Zero density assumption (ZDA)]
In the two-groups model in \eqref{eq:generic_two_groups_model}, it holds that $f_{\eta_1}(0) = (1-\eta_1)\phi(0)$. Equivalently, one of the following must be true: $\eta_1 = 0$ or $f_1(0)=0$.
\label{assumption:zerodensity}
\end{assumption}

It is important to recognize that while the two-groups interpretation is widely used for estimation, there exists a fundamental discrepancy between structure-agnostic approaches based on~\eqref{eq:generic_two_groups_model} and the structural model in~\eqref{eq:structural}. Without explicit consideration of the signal distribution and its relation with the response, as in~\eqref{eq:structural}, we have no way to deduce anything about the strength of individual signals. Moreover, when the `non-null' density $f_1$ in~\eqref{eq:generic_two_groups_model} is left unspecified, the resulting estimate need not correspond to any distribution arising from a signal convolution~$\phi\ast \P$.

Nonetheless, by a suitable reinterpretation, we show how to reconcile \eqref{eq:structural} with empirical practice under the zero density assumption in the two-groups model. Our interpretation focuses on symmetric signal distributions $\P$, with only Sections~\ref{subsec:two-groups1} and~\ref{sec:asymmetry-compatibility} extending beyond this symmetry constraint.
Our contributions are as follows. 

\color{black}
\begin{figure}
\centering
\begin{tikzpicture}[
    node distance=1.8cm,
    auto,
    thick,
    >={Stealth[]}
]
\node[circle, fill, inner sep=2pt] (start) {};
\node (nulls) [above right=1.5cm and 2cm of start] {`Null'};
\node (alternatives) [below right=1.5cm and 2cm of start] {`Non-null'};
\coordinate (end) at (6,0);
\node (f0) at (end |- nulls) {$f_0(y)=\phi(y)$};
\node (f1) at (end |- alternatives) {$f_1(y)$};
\node (marginal) [right=5cm of start] {$f_{\eta_1}(y)$ as in~\eqref{eq:generic_two_groups_model}};

\draw[->] (start) -- node[above left] {$1-\eta_1$} (nulls);
\draw[->] (start) -- node[below left] {$\eta_1$} (alternatives);
\draw[->] (nulls) -- (f0);
\draw[->] (alternatives) -- (f1);
\draw[->, dashed] (start) -- (marginal);
\end{tikzpicture}
\caption{Diagram of the structure agnostic two-groups model, similar to~\citet[Fig. 2.3]{efron2010largescale}.}
\label{fig:two-groups-diagram}
\end{figure}

\begin{itemize}[leftmargin=*]
    \item In Section~\ref{sec-two-groups}, we explain that mapping the structural model in~\eqref{eq:structural} onto the structure-agnostic two-groups model in~\eqref{eq:generic_two_groups_model} is an inherently ambiguous task. Unless $\P(X=0)=1$, there are infinitely many ways of expressing~\eqref{eq:structural} as a binary mixture. We focus on two specific ways. Section~\ref{subsec:two-groups1} describes the decomposition that takes the `null'/`non-null' terminology literally by setting $\eta_1 = \P(X \neq 0)$ and identifying `null' sites as precisely the ones with $X=0$. We call it the null/non-null two-groups model. Section~\ref{subsec:two-groups2} describes a decomposition that is extremal in the sense that it uses the smallest possible value of $\eta_1$ in~\eqref{eq:generic_two_groups_model}; see Theorem~\ref{theo:optim} for a precise statement. We call this the inactive/active two-groups model.  
    
    \item Section~\ref{sec:zero_assumptions_inactive} ties the null/non-null and inactive/active two-groups models to the zero density assumption. The former directly contradicts the zero density assumption, while the latter conforms to it. Empirical Bayes estimates of the local false discovery rate estimated under the zero density assumption effectively operate on the latter and we demonstrate this empirically in simulations. The implication is that such estimates of the local false discovery rate may not be interpreted as estimates of the local null-signal rate in~\eqref{eq:localnsr}.    
    \item In Section~\ref{sec-probabilistic-interpretation} we explain just how local false discovery rate estimates under the zero assumption are to be interpreted under the structural model in~\eqref{eq:structural}. Such an interpretation necessitates a careful probabilistic understanding of the inactive/active two-groups model. Denoting an active/inactive signal by a binary activity indicator $\mathcal{A}$, the relevant object that takes the role of the local false discovery rate is $\clar(y) = \P(\A = 0 \given Y)$, which we call the complement of the local activity rate. The activity indicator is not the same as the non-null signal indicator $1\{X\neq 0\}$, nor is it a step function $1\{|X| \ge c\}$ for any fixed threshold~$c \ge 0$.
    Our interpretation in Proposition~\ref{prop:hsec} is based on the hyperbolic secant identity $\clar(y) = \E\{ \sech(YX) \mid Y=y\}$. In Section~\ref{subsec:lfsr} we also show that $\clar(y)$ is at least as large as the local false sign rate \citep{stephens2017false}.
    \item Section~\ref{sec-sparse-limit-approximation} provides a further rule of thumb for interpreting $\clar(Y)$. Using the technical machinery of statistical sparsity~\citep{mccullagh2018statistical}, we show that $\clar(y)$ is approximately the conditional probability that the signal is less than $1/|y|$ in absolute value.
    \item Section \ref{sec:asymmetry-compatibility} extends our inactive/active two-groups model to the setting where the signal distribution is asymmetric, and contains a precise characterization of when the marginal distribution $\P * \phi$ is compatible with the zero density assumption. In particular, every (non-trivial) signal distribution $\P$ satisfying $\E\{X \mid Y=0\}=0$ is compatible with the zero density assumption, in the sense that there is some binary mixture representation of $\P * \phi$ satisfying $f_1(0)=0$.
    \item Section~\ref{subsec:ident_vs_estim} provides some comments on identifiability and practical estimability. The main point therein is that 
    even though the local null signal rate in~\eqref{eq:localnsr} is identifiable under the structural model~\eqref{eq:structural}, in general it cannot be estimated uniformly consistently without further assumptions. The implication is that a practitioner has two possible choices: either proceed with the zero density assumption, and adopt the interpretation presented herein, or to make further assumptions on the signal distribution and estimate the local null signal rate under these restrictions.
\end{itemize}

We emphasize that the goal of this paper is not to recommend new methods but to re-interpret existing methods. 
Many statisticians realize that conventional estimators of the local false discovery rate overestimate the local null signal rate in~\eqref{eq:localnsr}. In this paper, we formalize this mismatch and provide a precise characterization.

\section{A tale of two two-groups models}
\label{sec-two-groups}

In this section we ask: just how should one map the structural model in~\eqref{eq:structural} to the two-groups model in~\eqref{eq:generic_two_groups_model} and Fig.~\ref{fig:two-groups-diagram}? Implementing such a map, necessitates the specification of $\eta_1$ and $f_1$, such that
$(1-\eta_1) + \eta_1 f_1(y) = \smallint \phi(y-x) \P(\dd x)$, that is, the marginal density under the two-groups model must agree with the marginal density under~\eqref{eq:structural}. 
We define $\mathcal{H}$ as the set of all permissible choices of $\eta_1$:
\begin{multline}
\mathcal{H} := \{\eta_1 \in [0,1] : \text{ there exists a density } f_1 \text{ s.t.} \\
    (1-\eta_1)\phi(\cdot)+\eta_1 f_1(\cdot) = \smallint \phi(\cdot-x)\P(\dd x) \}.
\label{eq:all_two_groups_models}
\end{multline}

Given $\eta_1 \in \mathcal{H}\cap (0,1]$,\footnote{If $\eta_1=0$, then the choice of $f_1$ is immaterial.} the form of $f_1$ follows from the relationship~\citep{strimmer2008unified, patra2016estimation},
\begin{equation}
\label{eq:two_groups_to_f1}
f_1(y) = \{\smallint \phi(y-x)\P(\dd x) - (1-\eta_1)\phi(y)\}/\eta_1.
\end{equation}
As we will explain below, $\mathcal{H}$ is \emph{always} a non-degenerate closed interval. Thus, the task of mapping the structural model in~\eqref{eq:structural} to a two-groups models is not well-posed, and there are infinitely many ways of implementing such a decomposition. The statistical and inferential goal at hand will determine which two-groups model is most pertinent, and in what follows we highlight two specific choices.

\subsection{Structural null/non-null two-groups model}
\label{subsec:two-groups1}

As mentioned after~\eqref{eq:localnsr}, in most of the literature, \citep{mitchell1988bayesian,johnstone2004needles},
exact null signals play a special role. 
Following this tradition, in our first decomposition of~\eqref{eq:structural} as a two-groups model, we will take the interpretation of a `null' signal literally and decompose the signal distribution $\P$ as
\begin{equation}
\P  = (1-\pi_1)\delta_0 + \pi_1 \P(\cdot \mid X\neq 0),
\label{eq:signal_dbn}
\end{equation}
where $\delta_0$ is the Dirac point mass at $0$, and the non-null proportion is 
\begin{equation}
\label{eq:non-null-proportion}
\pi_1 := \P(X\neq 0) = \smallint 1\{x \neq 0\} \P(\dd x).
\end{equation}
Assuming that $\pi_1 > 0$, the marginal density for non-zero signals is
\begin{align*}
    m_1(y) &= \int \phi(x-y) \P(\dd x \mid X\neq0) \\
    &= \frac{1}{\pi_1}\int_{\Real\setminus \{0\}} \phi(y - x) \, \P(\dd x).
\end{align*}
The choices above, lead to one possible two-groups model that we call the structural null/non-null two groups model:
\begin{equation}
m_{\pi_1}(y) =  (1 - \pi_1) \phi(y) + \pi_1 m_1(y).
\label{mixture1}
\end{equation}
It follows that $\pi_1 \in \mathcal{H}$ defined in~\eqref{eq:all_two_groups_models}.
In our notation in~\eqref{mixture1}, we write $\pi_1$, $m_1(\cdot)$ and $m_{\pi_1}(\cdot)$, deviating from the notation using $\eta_1$, $f_1(\cdot)$ and $f_{\eta_1}(\cdot)$ in~\eqref{eq:generic_two_groups_model}. This will allow us to refer to different two-groups models without ambiguity.

The local null-signal rate, defined in~\eqref{eq:localnsr}, is equal to
\begin{equation}\label{lnsr}
\lnsr(y) = \frac{(1-\pi_1) \phi(y)} { m_{\pi_1}(y)},
\end{equation}
and satisfies $E\{\lnsr(Y)\} =1-\pi_1$.

As an example, consider the spike-and-slab Dirac-Cauchy signal distribution,
\begin{equation}
\label{eq:dirac_cauchy}
\P^{\delta C}_{\pi_1} = (1-\pi_1) \delta_0 + \pi_1  \mathrm{C}(0.5),
\end{equation}
where $\mathrm{C}(\sigma)$ is the Cauchy distribution with probable error $\sigma > 0$ and Lebesgue density $\sigma/\{\pi(\sigma^2 + x^2) \}$. In Figure~\ref{fig:two_groups}A, we plot the marginal density $m_{\pi_1}(y)$, for the signal distribution with $\pi_1 = 0.4$. We also plot the contribution $\pi_1 m_1(y)$ of the non-null signal component. 
From the figure, one may directly see that the structural null/non-null two-groups model in~\eqref{mixture1} does not satisfy the zero density assumption (Assumption~\ref{assumption:zerodensity}). This has been explicitly pointed out in the existing literature, e.g., by~\citet{genovese2004stochastic} and~\citet[Lemma 2]{jin2008proportion}. For self-containedness, we provide a formal statement. 

\begin{proposition}
\label{proposition:null_model_violates_zero_density}
Suppose that $\pi_1$ defined in~\eqref{eq:non-null-proportion} satisfies $\pi_1 > 0$. Then, the null/non-null two-groups model in~\eqref{mixture1} does not satisfy the zero density assumption (Assumption~\ref{assumption:zerodensity}), that is, $m_1(0)>0$.
\end{proposition}

\begin{proof}
Since $\P$ assigns a positive mass to $\Real\setminus \{0\}$, it follows that
$$ m_1(0) = \frac{1}{\pi_1} \int_{\Real\setminus \{0\}} \phi(x) \, \P(\dd x) > 0.$$
\end{proof}

\begin{figure}
\centering
\begin{tabular}{ll}
A  \\ 
    \includegraphics[width=0.45\textwidth]{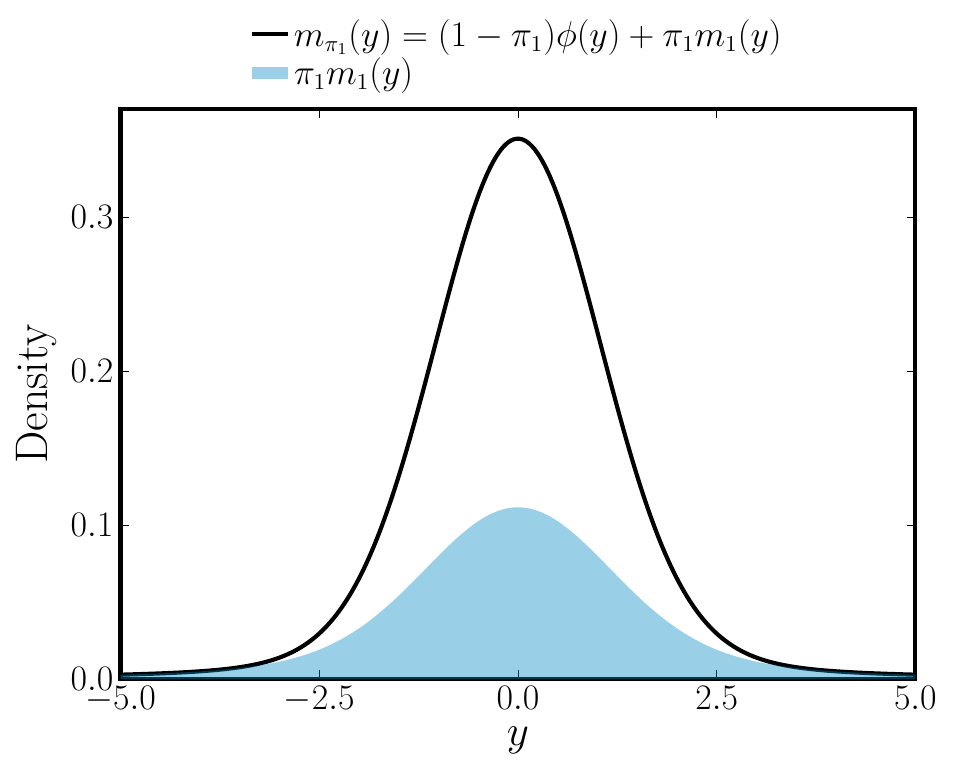} \\ B\\
    \includegraphics[width=0.45\textwidth]{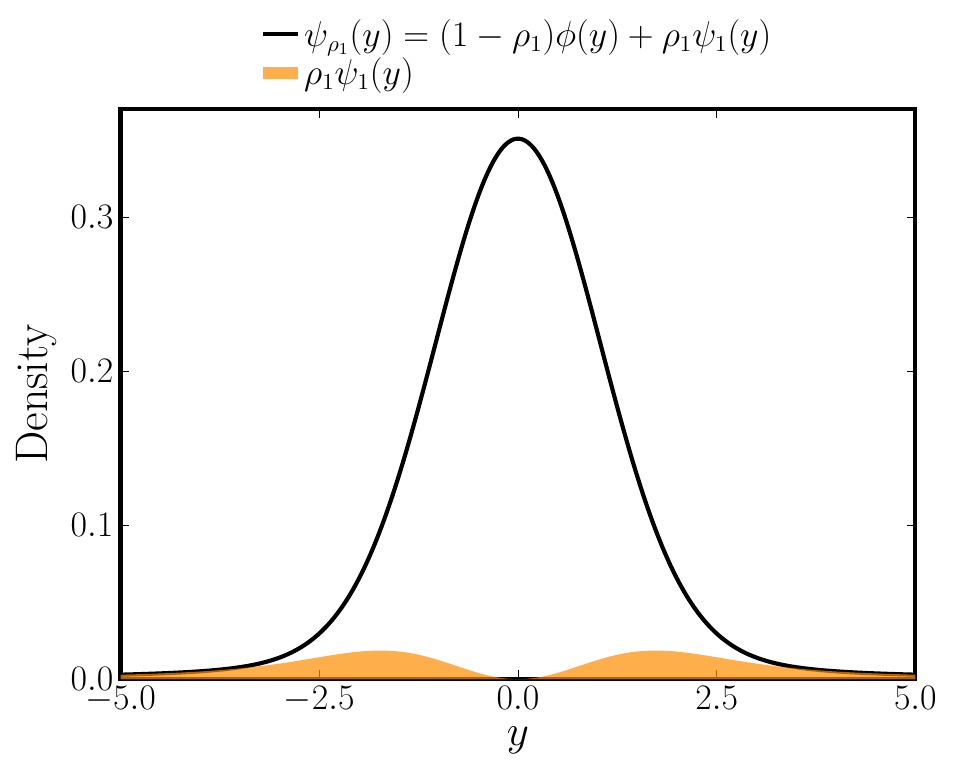} \\
\end{tabular}
\caption{Marginal density of $Y$ under the structural model~\eqref{eq:structural} with Dirac-Cauchy signal distribution $\P^{\delta C}_{0.4}$  in~\eqref{eq:dirac_cauchy}. Panel A also shows the contribution $\pi_1 m_1(y)$ of the non-null signal component of the two-groups model in~\eqref{mixture1} to the overall density. Meanwhile, Panel B shows the contribution $\rho \psi_1(y)$ of the active signal component of the alternative two-groups model in~\eqref{mixture2}. 
This population-level decomposition is qualitatively similar to the finite-sample version depicted in the bottom right panel of Figure 1 in \cite{stephens2017false}.}
\label{fig:two_groups}
\end{figure}

\subsection{Structural inactive/active two-groups model}
\label{subsec:two-groups2}

The structural null/non-null two-groups model in~\eqref{mixture1} appears to be a natural choice for a two-groups decomposition of the structural model in~\eqref{eq:structural}. It further pins down $\pi_1$ 
as one element of the set $\mathcal{H}$ in~\eqref{eq:all_two_groups_models} of possible two-groups decompositions. The fact that $\mathcal{H}$ contains further elements will not be surprising to the reader, for instance $1 \in \mathcal{H}$ via the (trivial) decomposition 
$$(1-1)\cdot \phi(y) + 1 \cdot \smallint \phi(y-x) \P(\dd x) = \smallint \phi(y-x) \P(\dd x).$$
Moreover, a convexity argument yields that $[\pi_1, 1] \subset \mathcal{H}$. It may be more surprising that unless $\pi_1=0$, $\mathcal{H}$ also contains elements $\eta_1 < \pi_1$ and that the precise boundary may be characterized as follows.

\begin{theorem}
\label{theo:optim}
    Let the signal distribution $\P$ be symmetric about the origin. Then $\mathcal{H}$ defined in~\eqref{eq:all_two_groups_models} is equal to the closed interval $[\rho_1,\, 1]$, where the lower boundary is
\begin{align}
\label{eq:rate}
\rho_1 := \int_{\mathbb R} ( 1 - e^{-x^2\!/2})\, \P(\dd x). 
\end{align}
\end{theorem}

\begin{proof}
By a convexity argument, it suffices to show the following: first, $\rho_1 \in \mathcal{H}$, and second, for any $\eta_1 < \rho_1$, it necessarily holds that $\eta_1 \notin \mathcal{H}$. To show that $\rho_1 \in \mathcal{H}$:
    \begin{align}
        &\int \phi(y - x) \, \P(\dd x) \,= \, \phi(y) \int (e^{x y} - 1+1)  e^{- x^2\!/2} \, \P(\dd x) \nonumber \\ \nonumber
   &\quad\quad \;\;\;\;\;\;\;\;\;\;\;\;\;\;\;\;\stackrel{(\star)}{=}  \phi(y) \int (\cosh(xy) - 1+1)  e^{- x^2\!/2} \, \P(\dd x) \\ 
    &\quad\quad \;\;\;\;\;\;\;\;\;\;\;\;\;\;\;\stackrel{(\star\star)}{=} (1 - \rho_1) \phi(y) +\rho_1 \psi_1(y).
   \nonumber
\end{align}
In $(\star)$ we use the symmetry of $\P$ and in $(\star \star)$ we define the density 
\begin{equation}
    \psi_1(y) := \rho_1^{-1}\phi(y)\int (\cosh(xy) - 1)  e^{- x^2\!/2} \, \P(\dd x).
    \label{eq:psi1}
\end{equation}
Conversely, let $\eta_1 < \rho_1$ and define $f_{1}$ as the density given via~\eqref{eq:two_groups_to_f1}. Then, we prove in the appendix that it must hold that $f_{1}(0) < 0$. Thus $\eta_1 \notin \mathcal{H}$.
\end{proof}
The boundary point $\rho_1$ defined in~\eqref{eq:rate} is called the sparsity rate 
by~\citet{mccullagh2018statistical}, a term motivated by replacing the indicator $x \mapsto 1\{x\neq 0\}$ in \eqref{eq:non-null-proportion} by the smooth function $x \mapsto 1-\exp(-x^2/2)$. See Section \ref{sec-background-sparsity} for a more detailed review of this notion of sparsity.

For our interpretation of the zero density assumption, the decomposition associated with the boundary point $\rho_1 \in \mathcal{H}$ is of special importance. We call it the structural inactive/active two-groups model and write
\begin{equation}
\psi_{\rho_1}(y) = (1 - \rho_1) \phi(y) +\rho_1 \psi_1(y),
\label{mixture2}
\end{equation}
where we use the following notation to distinguish it from the other two-groups models: $\rho_1$ as in~\eqref{eq:rate}, $\psi_1$ as in~\eqref{eq:psi1}, and $\psi_{\rho_1}$.

Revisiting the Dirac-Cauchy signal distribution 
$\P^{\delta C}_{\pi_1}$ in~\eqref{eq:dirac_cauchy} with $\pi_1=0.4$, we find that $\rho_1 \approx 0.12$. Figure~\ref{fig:two_groups}B shows the contribution $\rho_1 \psi_1(y)$ of the active signal component in the two-groups model~\eqref{mixture2} to the overall marginal density of $Y$. Inspecting the figure, we see that the zero density assumption holds as a consequence of the factor $\cosh(xy) - 1$ in~\eqref{eq:psi1}.
The next proposition formalizes this result and also shows that the mixture in~\eqref{mixture2} is the unique two-groups decomposition of the structural model in~\eqref{eq:structural} such that the zero density assumption holds.

\begin{proposition}
\label{proposition:inactive-active-satisfies-zero density}
Let the signal distribution $\P$ be symmetric about the origin. Among all elements $\eta_1 \in \mathcal{H}$ defined in~\eqref{eq:all_two_groups_models}, $\rho_1$ is the unique choice such that the associated two-groups model satisfies the zero density assumption (Assumption~\ref{assumption:zerodensity}). In particular, $\psi_1(0)=0$. 
\end{proposition}

\begin{proof}
Notice that $\cosh(0) = 1$. By~\eqref{eq:psi1}, it follows that $\psi_1(0)=0$.
Next let $\eta_1 > \rho_1$ and write $f_1$ for the density associated to $\eta_1$ in~\eqref{eq:all_two_groups_models} and $f_{\eta_1}$ for the marginal density as in~\eqref{eq:generic_two_groups_model}. By definition, $f_{\eta_1}(y) = \psi_{\rho_1}(y)$ for all $y$, and in particular $f_{\eta_1}(0) = \psi_{\rho_1}(0)$.
Since the two-groups model associated with $\rho_1$ satisfies the zero density assumption, $f_{\eta_1}(0) = (1-\rho_1)\phi(0) > (1-\eta_1)\phi(0)$. 
Expression \eqref{eq:two_groups_to_f1} implies $f_1(0)= \eta_1^{-1} (f_{\eta_1}(0) - (1-\eta_1)\phi(0))>~0$.
\end{proof}

The two components in the null/non-null two-groups model in~\eqref{mixture1} are associated with null signals ($X=0$) and non-null signals ($X\neq 0$) respectively. By analogy, we introduce an indicator $\A \in \{0,1\}$ for group membership in the inactive/active two-groups model~\eqref{mixture2}. We call $\A$ the activity indicator and call a signal active ($\A=1$) if its observation is drawn from $\psi_1(\cdot)$ and inactive ($\A=0$) if its observation is drawn from $\phi(\cdot)$. We postpone a more precise probabilistic definition, characterization and interpretation of $\mathcal{A}$ to Section~\ref{sec-probabilistic-interpretation}.

By analogy with the local null-signal rate in~\eqref{eq:localnsr}, we also define the complement of the local activity rate  
\begin{equation}\label{lar}
\clar(y) := \P(\A = 0 \given Y=y) = \frac{(1-\rho_1) \phi(y)} {\psi_{\rho_1}(y)}
\end{equation}
(\citet[Section 5.4]{mccullagh2018statistical}). We recall that $\psi_{\rho_1}(y) = m_{\pi_1}(y) = \smallint \phi(y-x)\P(\dd x)$. If $\pi_1 >0$,\footnote{If $\pi_1=0$, then $\clar(y) = \lnsr(y) = 1$ for all $y$.} then $\rho_1 < \pi_1$, and so
\begin{equation}
\label{eq:clar_bigger_than_lnsr}
\clar(y) > \lnsr(y) \text{ for all } y.
\end{equation}

In the next section (Section~\ref{sec:zero_assumptions_inactive}) we explain that common empirical Bayes strategies and software implementations based on the zero density assumption actually operate on~\eqref{mixture2} rather than~\eqref{mixture1}, and so provide estimates of $\clar(y)$ and not of $\lnsr(y)$.

\section{The zero density assumption targets inactivity rates}
\label{sec:zero_assumptions_inactive}

\begin{figure}
  \centering
\begin{tabular}{l}
   A \\
   \includegraphics[width=1.02\linewidth]{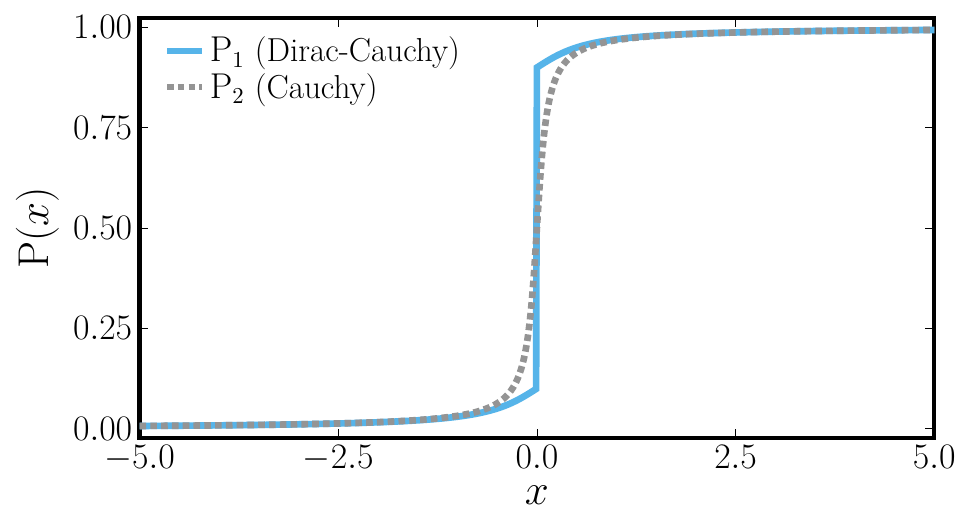} \\
   B \\ 
   \includegraphics[width=\linewidth]{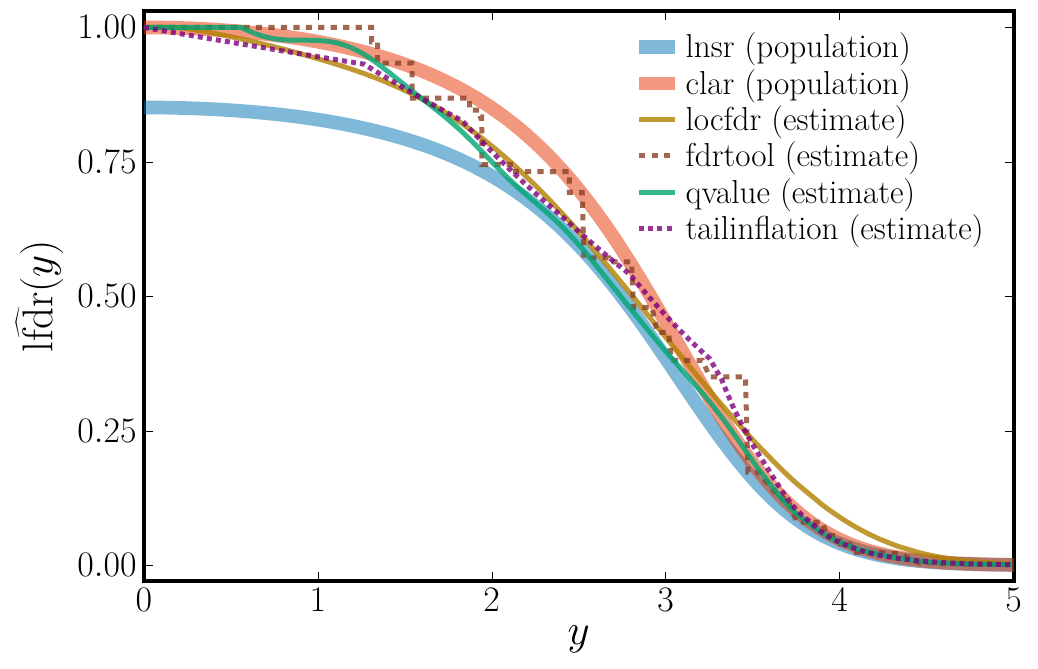} \\
   C \\ 
   \includegraphics[width=\linewidth]{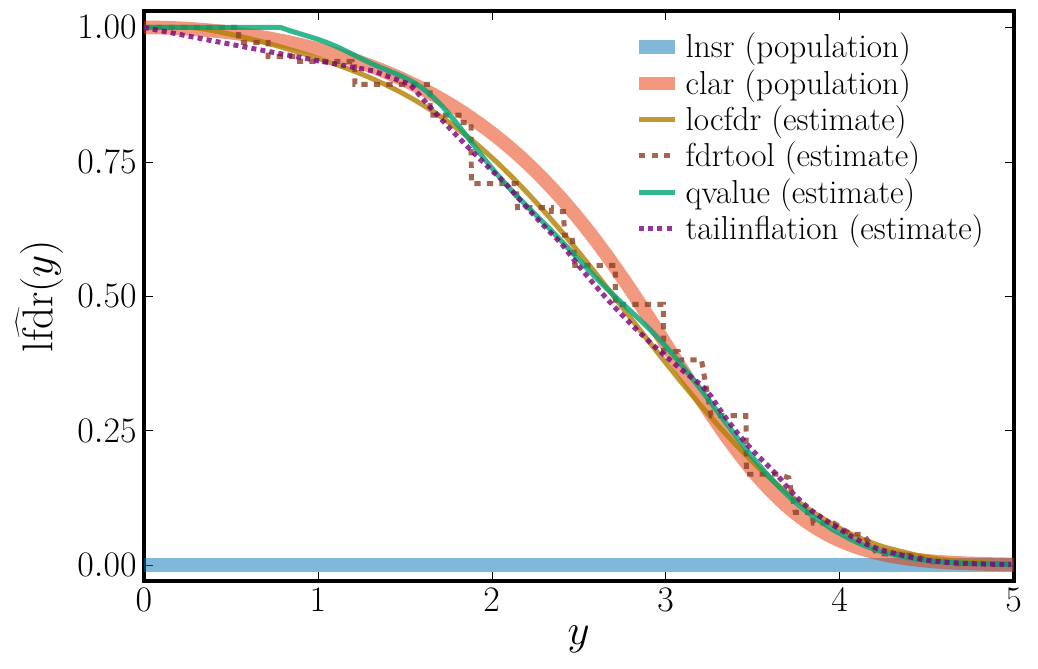} 
\end{tabular}
  \caption{Illustrative simulation. Panel A shows the distribution functions of the two signal distributions, $\P_1=\P_{0.2}^{\delta \text{C}}$ and $\P_2=\text{C}(0.1)$, defined in~\eqref{eq:dirac_cauchy}. Panel B (resp. C) shows the local null-signal rate, $\lnsr(y)$, and the complementary local activity rate, $\lidr(y)$, as a function of $y$ for the signal distribution $\P_1$ (resp. $\P_2$), as well as four estimates \smash{$\widehat{\lfdr}$} based on $10,000$ observations from the convolutional model~\eqref{eq:structural}. In panel C, \smash{$\lnsr_{\P_2}(y)=0$} for all $y$, but all estimates of \smash{$\widehat{\lfdr}$} are far from zero, and instead closely follow $\lidr(y)$.  }
  \label{fig:simple_lfdr_simulation}
\end{figure}

We now connect our preceding discussion to the zero density assumption (Assumption~\ref{assumption:zerodensity}) and to the most common implementations of local false discovery rate methodology in practice (and in existing software packages). The data consist of observations $Y_1,\dotsc,Y_n$ drawn from a large number of distinct units or sites. One posits that the marginal density of each $Y_i$ is given by $f_{\eta_1}(\cdot)$ as defined in~\eqref{eq:generic_two_groups_model} with $f_1(\cdot)$ and $\eta_1$ being unknown parameters to be estimated based on $Y_1,\dotsc,Y_n$.  

By the zero density assumption, $f_{\eta_1}(0) = (1-\eta_1)\phi(0)$, or equivalently, $\eta_1 = 1-f_{\eta_1}(0)/\phi(0)$. This expression motivates the following estimation strategy (as well as variations thereof) for $\eta_1$ and the `local false discovery rate'~\citep{efron2010largescale,genovese2004stochastic}: let \smash{$\widehat{f}(\cdot)$} be a direct estimate of the marginal density of $Y_1,\dotsc,Y_n$ (which could be parametric or nonparametric), and let
$$\widehat{\eta}_1 := 1-\frac{\widehat{f}(0)}{\phi(0)},\quad\ \widehat{\lfdr}(y) := \frac{(1-\widehat{\eta}_1) \phi(y)}{ \widehat{f}(y)}.$$

The estimated local false discovery rate $\widehat{\lfdr}(y)$ is then often interpreted as an estimate of the local null-signal rate $\lnsr(y)$~\eqref{eq:localnsr} in the null/non-null two-groups model in~\eqref{mixture1}.  However, as shown in Proposition~\ref{proposition:null_model_violates_zero_density}, unless $\pi_1=0$, the null/non-null two groups model does not satisfy the zero density assumption. Hence,  \smash{$\widehat{\lfdr}(y)$} may not be directly interpreted as an estimate of $\lnsr(y)$. Many statisticians are aware of this mismatch, and would only interpret \smash{$\widehat{\lfdr}(y)$} as a conservative (over)estimate of $\lnsr(y)$ 
\citep{storey2004strong,patra2016estimation}.

Here we show that a more precise interpretation is possible when $\P$ is symmetric. The key to the matter is that the inactive/active two-groups model in~\eqref{mixture2} satisfies the zero density assumption, as shown in Proposition~\ref{proposition:inactive-active-satisfies-zero density}. Thus, \smash{$\widehat{\eta}_1$} may be interpreted as an estimate of $\rho_1$ in~\eqref{eq:rate} and  \smash{$\widehat{\lfdr}(y)$} as an estimate of the complementary local activity rate $\clar(y)$ defined in~\eqref{lar}. Since $\clar(y) \geq \lnsr(y)$ (with equality only when $\pi_1=0$) as in~\eqref{eq:clar_bigger_than_lnsr}, the above statement is corcondant with the interpretation of \smash{$\widehat{\lfdr}(y)$} as a conservative (over)estimate of $\lnsr(y)$. In Sections~\ref{sec-probabilistic-interpretation} and~\ref{sec-sparse-limit-approximation} we build toward a more direct interpretation of the estimate of \smash{$\widehat{\lfdr}(y)$} by studying the complementary local activity rate $\clar(y)$.

We next consider a simulation to illustrate that commonly used software packages indeed furnish estimates of the complementary local activity rate instead of the local null-signal rate.
To this end, we consider the following two signal distributions: $\P_1 = \P^{\delta C}_{0.2}$, that is, the Dirac-Cauchy signal distribution in~\eqref{eq:dirac_cauchy} with $\pi_1 = 0.2$, and $\P_2 = \mathrm{C}(0.1)$. The cumulative distribution functions of $\P_1$ and $\P_2$ are shown in Figure~\ref{fig:simple_lfdr_simulation}A.
Despite the similarity of the signal distributions, their behavior in terms of the local null-signal rate is vastly different: $\lnsr(y)$ is identically zero for $\P_2$ and $\lnsr(0) \approx 0.85$ for $\P_1$. 

Continuing with this example, we simulate $n=10,000$ independent values $Y_1,\dotsc,Y_n$ from the convolutional model~\eqref{eq:structural} with signal distributions $\P_1$ and $\P_2$. We then compute estimates of the `local false discovery rate' \smash{$\widehat{\lfdr}(\cdot)$} using three widely-used software packages: \texttt{locfdr}~\citep{efron2015locfdr}, \texttt{fdrtool} \citep{strimmer2008fdrtool, klaus2021fdrtool}, and \texttt{qvalue}~\citep{storey2023qvalue}. The estimation strategies underlying these methods are distinct,\footnote{For estimating the marginal density, \texttt{locfdr} uses flexible parametric modeling through Lindsey's method~\citep{efron1996using}, \texttt{fdrtool} uses nonparametric shape constraints by imposing that the density of the two-sided p-values be non-increasing~\citep{grenander1956theory}, and \texttt{qvalue} uses smoothing splines.} yet they all build on the zero density assumption. We also consider one more  strategy that uses the zero density assumption along with an estimate \smash{$\widehat{f}(\cdot)$} of the marginal density 
computed using the \texttt{tailinflation} code~\citep{dumbgen2021active} for maximum likelihood subject to log-convexity of the density ratio $\widehat{f}/\phi(x)$, which is guaranteed
under~\eqref{eq:structural} (see \citealt[Example 2.1]{dumbgen2021active}).

The furnished estimates are shown in Figures~\ref{fig:simple_lfdr_simulation}B and~\ref{fig:simple_lfdr_simulation}C   (for $\P_1$, respectively $\P_2$). Panel C makes the following apparent: all four methods are evidently not estimating the zero function, $\lnsr_{\P_2}(\cdot) \equiv 0$. Instead, all four methods estimate a conservative upper bound on the $\lnsr(y)$, which as explained in Section \ref{sec-probabilistic-interpretation}, is equal to $\clar(y)$.

\section{Probabilistic interpretation of activity}
\label{sec-probabilistic-interpretation}

Our first step towards building a probabilistic interpretation for the $\clar$ is the following identity.
\begin{proposition}[Hyperbolic secant identity]
\label{prop:hsec}
    If the signal distribution $\P$ is symmetric, then
    \begin{equation}
 \clar(y) = \E\{ \sech(YX) \mid Y=y\}.
 \label{eq:hyperbolicsecant}
    \end{equation}
\end{proposition}
\begin{proof}
Starting from the right hand side of~\eqref{lar}, we find that $\clar(y)$ is equal to,
\begin{align*}
    \frac{(1-\rho_1) \phi(y)} {\psi_{\rho_1}(y)} &\stackrel{(\star)}{=} \frac {\phi(y)} {\psi_{\rho_1}(y)}  \int_\Real \frac{e^{xy}} {\cosh(xy)} e^{-x^2\!/2} \, \P(\dd x) \\
    &= \frac 1 {\psi_{\rho_1}(y)} \int_\Real \frac{\phi(y - x)} {\cosh(xy)} \, \P(\dd x) \\
    &= \E\{ \sech(YX) \mid Y=y\}.
\end{align*}
In $(\star)$ we used the symmetry of $\P$.
\end{proof}
The hyperbolic secant function $\sech(t) = 1/\cosh(t) = 2/(e^t + e^{-t})$ is a symmetric and non-negative kernel with total mass $\pi$, plotted in blue in Figure \ref{fig:sech-plot}. 
On account of continuity, the random variable $\sech(XY)$ has posterior mean that is continuous in the signal distribution $\P$. This provides one explanation for why $\clar$ is more amenable to statistical inference than the $\lnsr$; it is robust to small changes in $\P$ such as the amount of mass directly on versus near the origin (see also Proposition~\ref{prop:clar_easy_to_estimate}).

In what follows, the activity indicator $\A$ is not a deterministic function of $X$. Rather, it is a latent random variable whose joint density with $X$ and $Y$ is:
\begin{align*}
    f(x,y,a) = \begin{cases}
        \phi(y-x)\P(\dd x) \sech(xy) \hspace{.5em} &\text{if }a=0 \\
        \phi(y-x)\P(\dd x) (1-\sech(xy)) &\text{if }a=1.
    \end{cases}
\end{align*}

Accordingly, 
\begin{align*}
    \P(\A=0 \mid Y=y,X=x)= \sech(xy).
\end{align*}
Taking the posterior expectation on both sides of the above, we find that this joint distribution is consistent with the hyperbolic secant identity,
\begin{align*}
    \P(\A=0 \mid Y=y) = \E\{ \sech(YX) \mid Y=y\}.
\end{align*}
Integrating further against $\psi_{\rho_1}(y)$ yields the complement of the mean activity rate
\begin{align*}
    \P(\A=0) = \E\{\sech(YX)\} = 1-\rho_1.
\end{align*}
It is instructive to consider two marginalizations of the triple $(X,Y,\A)$, namely that of $(Y,\A)$ and $(X,\A)$. First, the joint distribution of $Y$ and $\A$ is that of the usual two-groups model after the zero-assumption has been made:
\begin{align*}
\A \sim \text{Bernoulli}(\rho_1),\,\,\; Y \mid \A \sim (1-\A) \phi + \A \psi_1,
\end{align*}
where $\psi_1(0)=0$. For an observation $Y=y$, $\clar(y)$ is the posterior probability that it arose from the $\mathrm{N}(0,1)$ component. Secondly, marginalization over $Y$ in the joint distribution of $(X,Y,\A)$ gives
\begin{equation}
\label{eq:activity_given_x}
X \sim \P,\,\,\, \A \mid X \sim \text{Bernoulli}(1-e^{-X^2/2}).
\end{equation}
For the pair $(X,\A)$, exact null signals ($X=0$) satisfy $\A=0$. However, a non-null signal ($X\neq 0$) also has a positive probability of being inactive; this probability is decreasing in $|X|$.

This interpretation is not a matter of the authors' opinion nor is it a subjective matter in the sense of an interpretation of a Bayesian prior.  It is a mathematical interpretation forced solely by symmetry and the extreme choice of zero density for the non-null component in \eqref{eq:generic_two_groups_model}.

\subsection{Activity and local false sign rates}
\label{subsec:lfsr}

\cite{gelman2000type} and \cite{stephens2017false} raised concerns about testing the point null hypothesis $X=0$, suggesting instead to ask about the sign of $X$. 
The local false sign rate, defined by~\citet{stephens2017false} as,
\begin{equation}
\label{eq:lfsr}
    \text{lfsr}(y) := \min\left\{ \P(X\leq 0 \mid Y),\, \P(X\geq 0 \mid Y)\right\},
\end{equation}
naturally addresses whether we can recover the sign of $X$. The idea is that, if we guess that $X$ is negative when $\P(X\leq 0 \mid Y) > \P(X\geq 0\mid Y)$, then $\lfsr(Y) = \P(X\geq 0\mid Y)$ is the probability that our guess is incorrect. 

Our main result in this subsection is that we can strengthen~\eqref{eq:clar_bigger_than_lnsr} as follows (as long as $\pi_1 >0)$:
\begin{equation}
\label{eq:clar_dominates_lfsr}
\clar(y) > \text{lfsr}(y) > \lnsr(y) \text{ for all } y.
\end{equation}
Thus, by controlling $\clar(y)$ we also control directional errors. Such a result is reminiscent of the directional false discovery rate control of the Benjamini-Hochberg procedure~\citep{benjamini2005false}.

We start with an identity analogous to the hyperbolic secant identity of Proposition~\ref{prop:hsec}.

\begin{proposition}
\label{prop:exp_identity}
For symmetric\footnote{The identity holds also for asymmetric $\P$ (see Section \ref{sec:asymmetry-compatibility}).} $\P$,
    \begin{equation}
 \clar(y) =  \E\{ \exp(-YX) \mid Y=y\}.
 \label{eq:exp}
    \end{equation}
\end{proposition}
\begin{proof} 
Arguing as in Proposition~\ref{prop:hsec}, we find that,
\begin{align*}
    \frac{(1-\rho_1) \phi(y)} {\psi_{\rho_1}(y)} &= \frac {\phi(y)} {\psi_{\rho_1}(y)}  \int_{\Real} e^{-xy} e^{xy} e^{-x^2\!/2} \, \P(\dd x) \\
    &= \frac 1 {\psi_{\rho_1}(y)} \int_\Real  e^{-xy} \phi(y - x) \, \P(\dd x) \\
    &= \E\{ \exp(-YX) \mid Y=y\}.
\end{align*}
\end{proof}
In classification, e.g., by AdaBoost, $u\mapsto \exp(-u)$ is used as a surrogate loss function of the $0/1$ loss $u \mapsto 1(u \leq 0)$~\citep{freund1997decision, friedman2000additive}. It holds that $\exp(-u) \geq 1(u \leq 0)$ (with equality only at $u=0$) and so along with Proposition~\ref{prop:exp_identity}, it follows that:
\begin{equation}
\P( XY \leq 0 \mid Y=y) \leq \clar(y).
\label{eq:modified_lfsr}
\end{equation}
Moreover, for a symmetric signal distribution, $\P( XY \leq 0 \mid Y=y) = \lfsr(y)$ for all $y \neq 0$, and noting that $\clar(0)=1$, the conclusion in~\eqref{eq:clar_dominates_lfsr} follows.

\color{black}

\begin{figure}
  \centering
  \includegraphics[width=\linewidth]{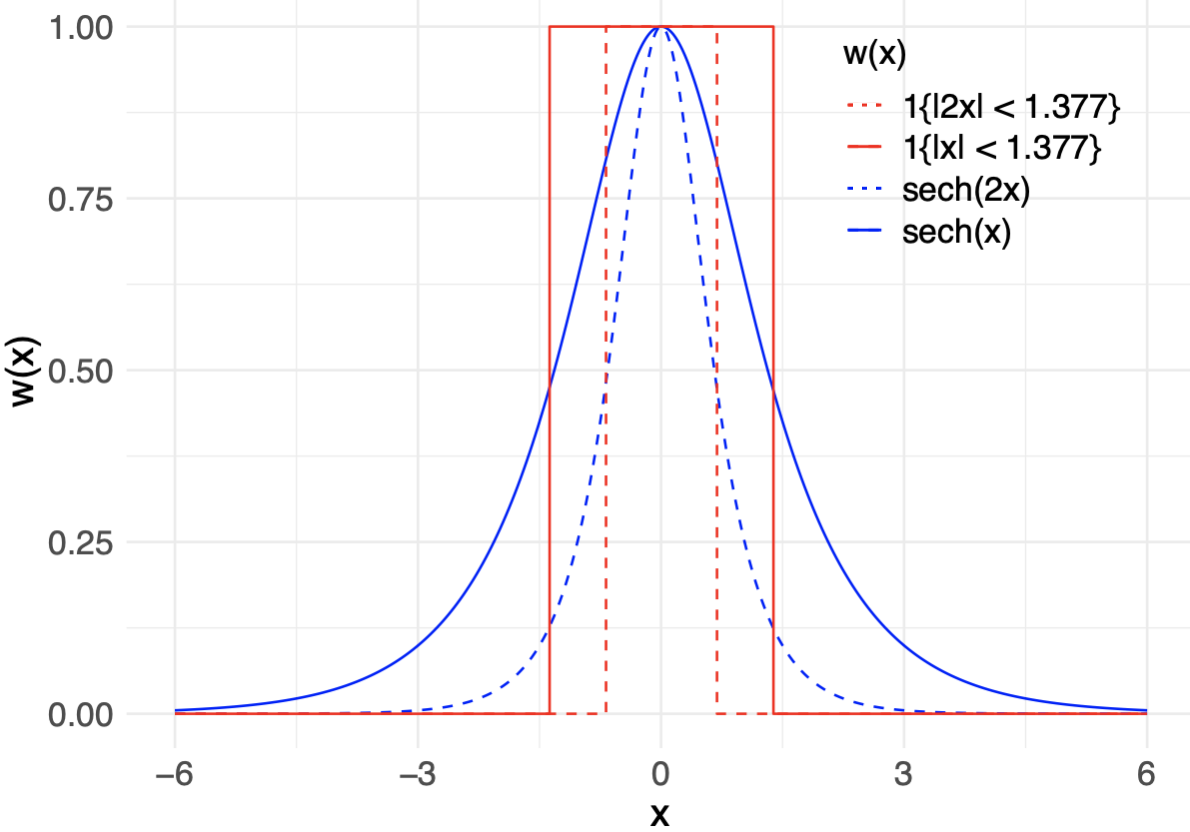}\\
  \caption{The hyperbolic secant function $\sech(xy)$ is plotted (blue) with an indicator function approximation $1\{|x| < \gamma/y\}$ overlaid (red) for $y=1,2$. As $y$ doubles, the range over which the hyperbolic secant function is non-negligible gets halved. } 
  \label{fig:sech-plot}
\end{figure}

\section{Sparse-limit approximation}
\label{sec-sparse-limit-approximation}

In the triple $(X,Y,\A)$, the activity variable $\A$ has a probabilistic definition given $(X,Y)$, 
\begin{align*}
    \mathcal{A}\mid (X,Y) \sim \text{Bernoulli}(\sech(XY)).
\end{align*}
The probabilistic interpretation of activity is less clear-cut than that of null / non-null signals, whose status is indicated by the event $\{X \neq 0\}$.
Sparsity is a limiting operation that enables a threshold approximation for the clar as follows.
For some data-dependent false-signal threshold $\delta(y)>~0$, 
\begin{align}
\label{eq:central-interval}
    \clar(y) \simeq \P\big( |X| \leq \tol(y) \mid Y=y \big),
\end{align}
where the approximation improves as $\rho_1$ becomes small. According to this interpretation, a signal is `active' if $|X|>\delta(y)$ and `inactive' otherwise.

The definition of sparsity used to obtain the threshold-approximation \eqref{eq:central-interval} is motivated by weak convergence of probability measures. The $\clar(y)$ is weakly continuous in the signal distribution $\P(\dd x)$, in the sense that it depends on $\P(\dd x)$ only through the functionals:
\begin{equation}
\label{eq:signal-functionals}
    \begin{aligned}
    \rho_1(\P) &= \int (1-e^{-x^2/2}) \P(\dd x) \\
    \psi_{1}(\P\; ;y) &= \rho_1^{-1}\phi(y)\int (\cosh(xy)-1) e^{-x^2/2} \P(\dd x),
\end{aligned}
\end{equation}
which are continuous with respect to small changes in $\P(\dd x)$, as measured in any reasonable probability metric. For example, if $\P_1 \approx \P_2 = \mathrm{C}(0.1)$ as in Section \ref{sec:zero_assumptions_inactive} where $\P_1$ has an atom at 0, then $\rho_1(\P_1)$ and $\rho_1(\P_2)$ are both between $6\%$ and $8\%$, and the $\clar(y)$ for $\P_1$ is close to its value for $\P_2$. In contrast, $\lnsr(y)$ depends sensitively on the atom at zero
(Panels B and C of Figure \ref{fig:simple_lfdr_simulation}), and gives very different answers for $\P_1$ and $\P_2$.

The two functionals in \eqref{eq:signal-functionals} determine the marginal density
\begin{align*}
    \psi_{\rho_1} = (1-\rho_1)\phi + \rho_1 \psi_1,
\end{align*}
and the sparsity rate~$\rho_1$ quantifies the difference between the distribution of $Y$ and $\mathrm{N}(0,1)$. 
From this perspective, $\P(\dd x)$ can be sparse even when it places no probability mass directly on the origin. A similar notion of sparsity was considered by \cite{rockova2018bayesian}, who proposed a continuous-spike and slab distribution
\begin{align*}
        X \sim (1-\theta)L(\lambda_0) + \theta L(\lambda_1),
    \end{align*}
    where $L(\lambda)$ is the scaled Laplace distribution with density 
    \begin{align}
    \label{eq:laplace}
        f(x) = \frac{1}{2\lambda} e^{-|x|/\lambda}
    \end{align}
    and $\theta \in (0,1)$ is the mixing probability.
In their framework, an `inactive' signal is defined by whether or not $|X|$ exceeds a fixed threshold that depends only on the signal distribution. Several other works that also consider a false-signal threshold are discussed in Section \ref{sec-discussion}.

In this section, we show that when the signal is sparse in the \cite{mccullagh2018statistical} sense, the false-signal threshold $\delta(y)$ in \eqref{eq:central-interval} is inversely proportional to $|y|$. 
To be concrete, for the Cauchy distribution $\mathrm{C}(\sigma)$ with probable error $\sigma$, our results imply 
\begin{align}
\label{asymp-equiv}
    \P_{\mathrm{C}(\sigma)}\left( |X| \leq 1.3770/|y|\,\mid \, Y=y\right)/\clar(y) \to 1,
\end{align}
as $\sigma \to 0$, holding $\clar(y)$ fixed between zero and one. The multiplicative factor $\gamma_1 := 1.3770$ arises from an approximation of the conditional probability $\P(|X|\leq \delta(y) \mid~Y)$ using the sparse-limit technique, recorded in Section \ref{sec-proofs}.

\subsection{Background on sparsity}
\label{sec-background-sparsity}

The definition of a sparse limit by \cite{mccullagh2018statistical} characterizes the limiting behavior of a sequence of distributions $\P_\nu \to \delta_0$ as mass accumulates near the origin. To motivate the definition, consider the probability that the $\mathrm{Cauchy}(0,\sigma)$ distribution places beyond some fixed threshold $\gamma>0$. This probability scales proportionally with $\sigma$,
\begin{align*}
    \P_{\mathrm{C}(\sigma)}([\gamma,\infty)) \sim \frac{\sigma}{\pi \gamma},
\end{align*}
as $\sigma \to 0$. Dividing this probability by the sparsity rate,
\begin{align*}
    \rho_1 = \int (1-e^{-x^2/2})\P_{\mathrm{C}(\sigma)}(\dd x) \sim \sigma\sqrt{2/\pi},
\end{align*}
yields a non-zero limit
\begin{align*}
    \rho_1^{-1} \P_{\mathrm{C}(\sigma)}([\gamma,\infty)) &\sim 
 \frac{1}{\sqrt{2\pi}\gamma} 
 = \int_\gamma^\infty \H_1(\dd x),
\end{align*}
where $\H_1(\dd x) := \frac{\dd x}{\sqrt{2\pi} x^2}$. 
The definition of a sparse limit requires the above convergence to hold not just for the indicator test function $1\{[\gamma,\infty)\}$, but for all regular functions of order $x^2$ near the origin. For a signal distribution $\P$ to have a sparse limit, the rescaled measure $\rho_1^{-1} \P$ must converge weakly to an exceedance measure $\H$ in the following sense.
\begin{definition}[\cite{mccullagh2018statistical}]
\label{def:sparsity}
    A family of distributions $\{\P_\nu\}$ indexed by $\nu > 0$ has a sparse limit with rate $\rho_1 = \rho_1(\nu) > 0$ and unit exceedance measure $\H$ if
\begin{equation}\label{eq:exceedance_limit}
\lim_{\nu \to 0} \rho_1^{-1} \int_\Real w(x)\, \P_\nu(\dd x) = \int_\Real w(x) \, \H(\dd x) < \infty
\end{equation}
for every bounded continuous function $w$ such that $|x|^{-2} w(x)$ is also bounded, and 
$\int(1 - e^{-x^2\!/2}) \H(\dd x)=~1$.
\end{definition} 
For example, the test function
\begin{align*}
    w(x) &= (\cosh(xy)-1)e^{-x^2/2}
\end{align*}
leading to the zero density component in \eqref{eq:signal-functionals} is continuous, bounded and behaves as $O(x^2)$ near the origin.
Sparsity requires that any pair of sparse signal sequences with the same exceedance measure give the same value for the integral of $w(x)$ 
as $\rho_1 \to 0$. This ensures that signal distributions within the same sparsity class give rise to approximately the same marginal density for $Y$.

The following examples exhibit a range of sparse families and in each case we state the pair $(\rho_1,\H)$ that characterizes the approach of $\P_\nu$ to a point mass at the origin as $\nu \to 0$.

\begin{example}
The first four examples are sparse in the sense of Definition~\ref{def:sparsity}.
\begin{enumerate}[leftmargin=15pt]
    \item The sparse Cauchy family $\P_\nu = \mathrm{C}(\nu)$ has limiting rate $\rho_1 = \nu \sqrt{2/\pi}$ for small~$\nu$, and the normalized exceedance density is $x^{-2}/\sqrt{2\pi}$.
    \item The atom-and-slab family $(1-\nu)\delta_0+ \nu \mathrm{C}(1)$ has rate $\rho_1 \simeq 0.477 \nu$ and normalized exceedance measure $2.097 \mathrm{C}(1)$.
    \item For every $\beta > 0$, the spike-and-slab family $(1-\nu)N(0, \nu^{1+\beta}) + \nu \mathrm{C}(1)$ is sparse with the same rate and exceedance measure as in the previous example. 
     \item The Student-$t$ distribution on $0<\q <2$ degrees of freedom with scale parameter $\nu\to 0$ has a sparse limit with rate and inverse-power exceedance measure
    \begin{equation}
        \begin{aligned}
            \label{eq:inverse-power-H}
        \rho_1 &= \frac{\q^{\q/2}\Gamma\left(\frac{\q+1}{2} \right)}{C_\q\sqrt{\pi}\Gamma(\q/2)} \, \nu^\q, \hspace{1em} \frac{\dd \H_\q}{\dd x} =  \frac{C_\q}{|x|^{\q+1}}, 
        \end{aligned}
    \end{equation}
    where $C_\q := \q\, 2^{\q/2-1}/\Gamma(1-\q/2)$.

\item
The Laplace mixture $(1-\nu)L(\nu^{1/2}) + \nu L(1)$ is a continuous spike-and-slab mixture in the sense of \cite{rockova2018bayesian}, where $L(\lambda)$ is defined in \eqref{eq:laplace}.  The rate integral \eqref{eq:rate} is well-defined and equal to $1.344\nu$ for small~$\nu$.
But the spike contribution is $O(\nu)$ and not negligible, which means that there is no exceedance measure satisfying~\eqref{eq:exceedance_limit}.

\end{enumerate}

\end{example}

\subsection{Conditional probability for a central interval}\label{sec:sparselimit}

As a function of $y$, we show that the threshold in \eqref{eq:central-interval} takes one of two forms depending on whether the exceedance measure is infinite or finite:
\begin{equation}
\label{eq:tol_form}
\arraycolsep=5pt
\tol(y) = \left\{\!\! \begin{array}{cl} \gamma_\alpha / |y| & (0 < \alpha < 2); \\ |\alpha|\log|y| / |y| & (\alpha < 0) .\end{array}
\right.
\end{equation}
A negative activity index\footnote{also known as the regular variation index} $\alpha<0$ implies a finite exceedance measure, and $\gamma_\alpha$ is a constant depending on the behaviour of the exceedance measure near the origin. For example, as already previewed in~\eqref{asymp-equiv}, under the sparse Cauchy model $\P_{\sigma} = \textsc{C}(\sigma)$, $\sigma \to 0$, $\delta(y)$ is characterized by the first branch of~\eqref{eq:tol_form} with $\alpha =1$. As a second example, the sparse atom-and-slab mixture $\P_\nu = (1 - \nu) \delta_0(\de x) + \nu \textsc{C}(1)$ with $\nu \to 0$ has a finite exceedance measure with index $\alpha=-1$ and so $\delta(y)$ is characterized by the second branch in~\eqref{eq:tol_form}.

Under general sparsity assumptions, we show in Theorem \ref{thm-constant-thresholds} that the limiting threshold satisfying \eqref{eq:central-interval} is inversely proportional to~$|y|$. 

\begin{theorem}
\label{thm-constant-thresholds}
    Suppose that $\P_\nu(\de x)$ has a sparse limit with exceedance $H$ and rate $\rho_1 \to 0$ as $\nu \to 0$. For any $\gamma,\omega > 0$, we have
    \begin{align}
    \label{eq:lfdr-equivalence}
        \lim_{\nu \to 0}\{ \P_\nu(|X| < \gamma/|y| \mid Y=y)/ \lidr(y)\} = 1,
    \end{align}
    for any $y \equiv y(\rho_1)$ such that $\rho_1 \zeta_\nu(y) = \omega$ and $|y| = o(\rho_1^{-1/2})$, where $\zeta_\nu(y):= \rho_1^{-1}\int_{\Real} (\cosh(xy) - 1)  e^{- x^2\!/2} \, \P_\nu (dx)$.
\end{theorem}
In one version of this limit, $y$~is fixed as $\rho_1 \to 0$, and $\lidr(y) \to 1$. In the less trivial version, $\rho_1 \to 0$ and $\rho_1 \zeta_\nu(y) = \omega > 0$ is held fixed, so that $\lidr(y)$ has a limit $1/(1 + \omega)$ strictly less than one, and $y$ is of order $O(-\log\rho_1)$. It suffices here to consider only the second version, with $\omega = 0$ corresponding to the trivial case. 

The limiting statement \eqref{eq:lfdr-equivalence} holds for any $\gamma>0$, and in that sense Theorem \ref{thm-constant-thresholds} doesn't identify a unique threshold. A more stringent condition is
\begin{align}
\label{cond:exact-threshold}
    \frac{\P_\nu (|X| < \tol(y) \mid Y=y)}{\clar(y)} - 1 = o(\rho_1).
\end{align}
The exact formula for $\tol(y)$ for which the \eqref{cond:exact-threshold} holds is calculated for a range of exceedance measures $\H(\dd x)$, as summarized in Table \ref{tab:formula-threshold}.

A precise statement of these results can be found in the \hyperref[appn]{Appendix} (Proposition \ref{prop-threshold-form}).
\begin{table}[htbp]
    \centering
    \arraycolsep=8pt
    \renewcommand{\arraystretch}{2} 
    \begin{tabular}{ccccc}
        \( \H(\dd x) \) &  &  & \( \tol(y)  \) & Range\\ 
        \hline 
        $\displaystyle{\frac{\de x} {|x|^{1+\alpha}}}$ &  
        &  & \(\frac{\gamma_\alpha}{|y|}\) & \(0 < \alpha < 2\)\\
        \( |x|^{-1} e^{-\beta|x|}\, dx \) &  &  & \( \frac{\log\log |y|}{|y|}\)  &\(\alpha=0\)\\ 
        $\displaystyle{\frac{\de x}{1+x^2}}$ & & 
        &\( \frac{\log |y|}{|y|}\) & \\ 
        \( |x|^{\alpha-1}e^{-\beta|x|}\,\de x \) & &
        & $ \frac{\alpha \log |y|}{|y|}$ & \(\alpha ,\beta > 0\) \\
        \noalign{\medskip}
        \hline
    \end{tabular}
    \caption{The second column is the threshold $\tol(y)$ satisfying \eqref{cond:exact-threshold} for a range of signal behaviors near the origin.}
    \label{tab:formula-threshold}
\end{table}

The sparse-limit false-discovery threshold for the inverse-power family \eqref{eq:inverse-power-H} is $\tol(y)=\gamma_\alpha/|y|$, where $\gamma_\alpha$ is a constant defined as the solution to the equation $\alpha J_\alpha(\gamma_\alpha)=1$, where $J_{\alpha}$ is a function arising from a series expansion of $\cosh$, 
\begin{align*}
    J_\alpha(x) := \sum_{r=1}^\infty \frac{x^{2r}} {(2r)!\, (2r-\alpha)}.
\end{align*}
Details regarding $J_\alpha$ and its relevance to equation \eqref{cond:exact-threshold} are provided in the \hyperref[appn]{Appendix}. 
The constant $\gamma_\alpha$ is shown for selected values of $\alpha$:
\[\arraycolsep=5pt
\begin{array}{cccccccc}
\alpha &  0.50 & 0.75 & 1.0 & 1.25 & 1.50  \\
\gamma_\alpha &  2.2370 & 1.7383 & 1.3770 & 1.0809 & 0.8120 
\end{array}
\]

The central interval interpretation in general depends on $\alpha$. In principle, we could try to estimate $\alpha$ as in e.g.,~\citet{tresoldi2024sparselimit}. Since in this paper, we do not advocate for any new methodology per se, and since Cauchy or Horseshoe priors~\citep{carvalho2010horseshoe} (which correspond to the index $\alpha=1$) are commonly used in practice, we recommend using the threshold $\delta(y) = 1.377/|y|$ as a rule of thumb for guiding interpretation.

\section{Asymmetry and compatibility}
\label{sec:asymmetry-compatibility}

In this section, we seek to extend our interpretation to asymmetric signal distributions. The first question that arises is that of compatibility: a signal distribution $\P$, not necessarily symmetric, is said to be compatible with the zero density assumption if the marginal distribution of $Y$ admits a two-groups decomposition satisfying this assumption. The following result provides a sharp characterization of when a signal distribution is compatible with the zero density assumption.

\begin{theorem}
\label{thm-compatibility}
    A signal distribution $\P$ is compatible with the zero density assumption if and only if $\int x e^{-x^2/2}\P(\dd x) = 0$, which is equivalent to $\E\{X \mid Y=0\}=0$.\footnote{Here, $\E\{X \mid Y=y\}:= \int x\phi(y-x)\P(\dd x) \big/ \int \phi(y-x)\P(\dd x)$.} In particular, every symmetric distribution is compatible with the zero density assumption.
\end{theorem}

\begin{proof}
    Since the exponential function $y \mapsto \exp(xy)$ is strictly convex for every $x \neq 0$, the density ratio,
    \begin{align*}
        h(y) := \frac{\int \phi(y-x)\P(\dd x)}{\phi(y)} =  \int e^{xy-x^2/2}\P(\dd x)
    \end{align*}
    is positive and strictly convex for all $\P(\cdot) \neq \delta_0(\cdot)$. The minimum $h(y^*) \leq 1$ occurs at a point $y^*$, possibly infinite, such that $h'(y^*)=0$, i.e.
    \begin{align}
    \label{eq:zero-posterior-mean}
        \int xe^{xy^*-x^2/2}\P(\dd x) = 0,
    \end{align}
    or equivalently, $\E\{X \mid Y=y^*\}=0$. The marginal distribution has a unique decomposition
    \begin{align*}
        \phi(y) h(y) &= \phi(y) h(y^*) + \phi(y)\big\{h(y)-h(y^*)\big\} \\
        &= (1-\rho_1) \phi(y) + \rho_1 \psi_{1}(y),
    \end{align*}
where $\rho_0 := h(y^*)$ and 
\begin{align}
\label{def:compatible-active-component}
    \psi_1(y) &:= \rho_1^{-1} \phi(y) \int (e^{xy}-e^{xy^*})e^{-x^2/2} \P(\dd x).
\end{align}
Since $h$ is strictly convex and $\phi$ is strictly positive, $\psi_1$ is zero at $y^*$ and strictly positive elsewhere. 
Therefore, if $\psi_1(0)=0$, then $y^*=0$ and \eqref{eq:zero-posterior-mean} implies $\int xe^{-x^2/2}\P(\dd x)=0$. 

On the other hand, if $\int xe^{-x^2/2}\P(\dd x)=0$, then 
\begin{align*}
    \psi_1(y) &= \rho_1^{-1}\phi(y) \int( e^{xy}-1 )e^{-x^2/2}\P(\dd x) \\
    &= \rho_1^{-1}\phi(y) \int( e^{xy}-xy-1 )e^{-x^2/2}\P(\dd x) \geq 0
\end{align*}
is a valid probability density satisfying  $\psi_1(0)=0$, so $\P$ is compatible with the zero density assumption.

\end{proof}

If $\P$ is symmetric, then $y^*=0$ and the expression \eqref{def:compatible-active-component} recovers the previous formula for the active component \eqref{eq:psi1}. By Theorem \ref{thm-compatibility}, the two-groups decomposition $\psi_{\rho_1} = (1-\rho_1) \phi + \rho_1 \psi_1$ is also valid for asymmetric signal distributions $\P$ satisfying $\E\{X \mid Y=0\}=0$.

This two-groups decomposition is illustrated in Figure \ref{fig:two_groups_asymmetric} for the following three point signal distribution:
\begin{equation}
\P^{\text{3pt}} = 0.8\delta_0 + 0.12 \delta_2 + 0.08 \delta_{\mu},
\label{eq:three_point}
\end{equation}
where $\mu \approx -0.45$ is such that $\E_{\P^{\text{3pt}}}\{X \mid Y=0\}=0$. In this case, $\P$ is asymmetric and so is $\psi_1(y)$. However, $\psi_1(0)=0$, that is, the zero density assumption holds. 

For a signal distribution $\P$ that is compatible with the zero density assumption (but not necessarily symmetric), we can thus also introduce an indicator $\A$ such that $Y \mid \A =1 \sim \psi_1(\cdot)$ corresponds to a draw from the component $\psi_1$. Thus we may also define $\clar(y)$ precisely as in~\eqref{lar}. Moreover, estimation strategies based on the zero density assumption, as described in Section~\ref{sec:zero_assumptions_inactive} furnish estimates of $\clar(y)$. 

However, the interpretation of $\clar(y)$ is not as straightforward for compatible but asymmetric signal distributions. The hyperbolic secant identity of Proposition \ref{prop:hsec} is not applicable, and the probabilistic definition of activity $\P(\A=0 \mid X,Y) = \sech(XY)$ is inconsistent with the two groups model $(1-\rho_1) \phi + \rho_1 \psi_1$, where $\psi_1$ is defined as in \eqref{def:compatible-active-component}. A joint distribution for $(X,Y,\A)$ that is consistent with this two-groups model in the case where $\P$ is asymmetric but compatible is constructed in Appendix \ref{subsec:compatible-extension}. We note that although the hyperbolic secant construction fails for the asymmetric case, the identity in Proposition \ref{prop:exp_identity} still holds as well as the inequality $\P( XY \leq 0 \mid Y=y) \leq \clar(y)$ in~\eqref{eq:modified_lfsr}. The quantity $\P( XY \leq 0 \mid Y=y)$ was used as a notion of local false sign rate by~\citet{ignatiadis2022confidence}, however as noted by~\citet{xie2022discussion}, it is different than the notion of local false sign rate in~\eqref{eq:lfsr}.\footnote{These notions agree for symmetric signal distributions $\P$ and $y \neq 0$, but not for asymmetric $\P$.}

\color{black}

\begin{figure}
\includegraphics[width=0.45\textwidth]{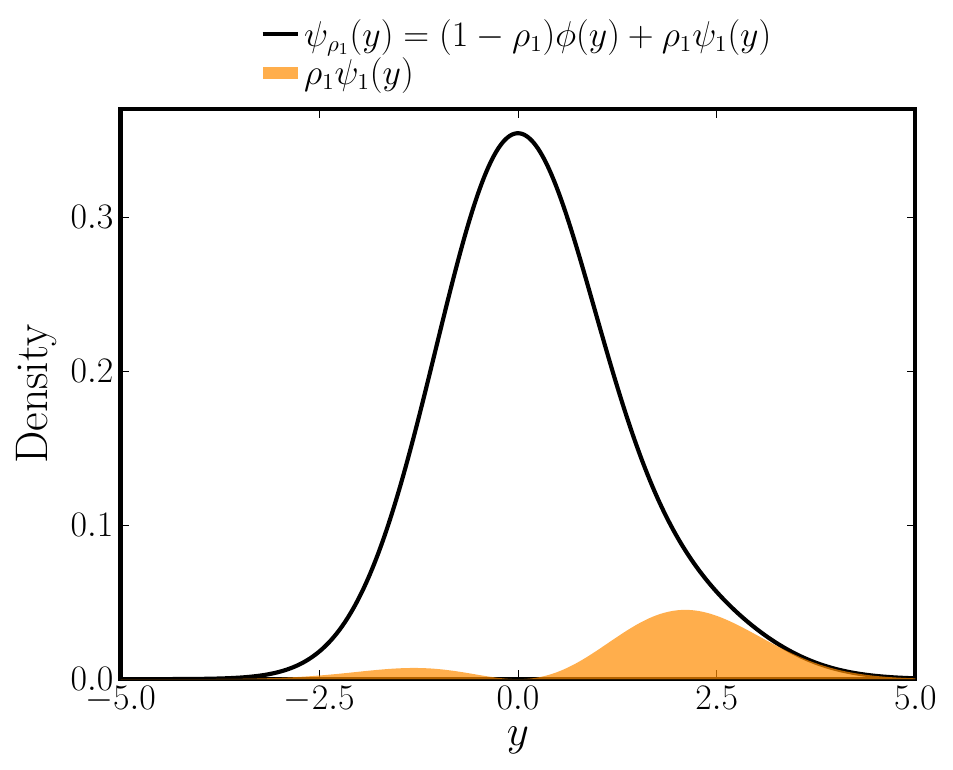} 
\caption{Marginal density of $Y$ under the structural model~\eqref{eq:structural} with the asymmetric signal distribution $\P^{\text{3pt}}$  in~\eqref{eq:three_point}. The plot also shows the contribution $\rho \psi_1(y)$ of the active signal component of the alternative two-groups model in~\eqref{mixture2}. The active component $\psi_1$ is asymmetric and satisfies the zero density assumption $\psi_{1}(0)=0$.}
\label{fig:two_groups_asymmetric}
\end{figure}

\section{Estimating the local null-signal rate}
\label{subsec:ident_vs_estim}

Suppose we have iid observations $Y_i$ drawn from the structural model in~\eqref{eq:structural} where the signal distribution $\P$ is unknown. 
So far we have argued the following: common empirical Bayes workflows building on the zero density assumption do not furnish estimates of the local null-signal rate $\lnsr(y)$ in~\eqref{eq:localnsr}, but instead of the complementary local activity rate, $\clar(y)$ in~\eqref{lar}. But what if substantive interest lies in estimating $\lnsr(y)$?

In principle, we can estimate $\lnsr(y)$. It is well-known \citep[Section 4]{teicher1960mixture} that the signal distribution $\P$ in the structural model~\eqref{eq:structural} is identifiable based on marginal observations $Y$. Since $\P$ is identifiable, it follows that $\pi_1$ and $m_1$ in the null/non-null two-groups model in~\eqref{mixture1} are also identifiable and thus also $\lnsr(y)$.\footnote{
Identifiability does not contradict the alternative decomposition into the inactive/active two-groups model in~\eqref{mixture2}. The reason is that $\psi_1$ in~\eqref{mixture2} may not be written as the convolution of a distribution $\P'$ with $\mathrm{N}(0,1)$, while $m_1$ in~\eqref{mixture1} can.
}
Identifiability of $\lnsr(y)$ however does not imply that it can be estimated based on finite data and the next proposition shows that uniformly consistent estimation of $\lnsr(y)$ is impossible.
\begin{proposition}
\label{prop:practically_unidentified}
Fix $y \in \R$, $0<\eta < 1/\surd{(2\pi)}$, and let $\mathcal{P}_{y,\eta} := \{ \P \text{ symmetric distribution}: \int \phi(y-x)\P(\dd x) > \eta\}$. 
Then, there exists $\varepsilon > 0$ such that:
$$\liminf_{n \to \infty} \inf_{\widehat{\lnsr}_n} \sup_{\P \in \mathcal{P}_{y,\eta}} \P( |\widehat{\lnsr}_n(y) - \lnsr(y)| > \varepsilon) > 0.$$
The infimum is taken over all measurable functions $(Y_1,\dotsc,Y_n) \mapsto \widehat{\lnsr}_n(y; Y_1,\dotsc,Y_n)$.
\end{proposition}

A proof is provided in Section \ref{sec-minimax-proof}.  Although we are not aware of a formal statement in the literature, the result of the proposition is well-known~\citep{carpentier2019adaptive}. The intuition is that after convolution with normal noise, a Dirac point mass at $0$ is statistically indistinguishable from a concentrated continuous spike around $0$. By contrast, $\clar(y)$ can be estimated with fast rates without any further assumptions, see Appendix~\ref{subsec:estimating_clar}. We view this as a strong argument for preferring $\clar(y)$ over $\lnsr(y)$, especially since $\clar(y)$ is already being estimated by common software packages.

We conclude this section by describing some options if $\lnsr(y)$ is indeed the quantity of interest.
\begin{itemize}[leftmargin=*]
    \item First, one may impose stronger assumptions on the signal distribution. As one example, one may assume that there exists a separation between null and non-null signals, i.e., that there exists $M>0$ such that $\P( |X| \geq M \mid X \neq 0) = 1$, or one may assume an analogous condition in the Fourier domain. If separation is sufficiently large (in a way that may depend on $n$), then uniformly consistent estimation of $\lnsr(y)$ is possible as shown by~\citet*{cai2007estimation} and~\citet{ jin2008proportion}.
    \item One may also make further assumptions on the non-null component $\P(\cdot \mid X \neq 0)$, e.g., that $\P(\cdot \mid X \neq 0)$ is known exactly~\citep{johnstone2004needles}, takes on a parametric~\citep{muralidharan2010empirical} or semiparametric form~\citep{martin2012nonparametric}, or obeys certain shape constraints such as unimodality~\citep{stephens2017false}.
    \item One may also still aim at conservative inference of $\lnsr(y)$ by suitable regularization as in~\citet{stephens2017false, xie2022discussion}  or by constructing lower confidence intervals for $\pi_1$ (and thus upper confidence intervals for $\lnsr(y)$) as in~\citet{meinshausen2006estimating, ignatiadis2022confidence, xue2023nonparametric}.\footnote{The approach in~\citet{meinshausen2006estimating} yields intervals that include $\rho_1$, respectively $\clar(y)$, with high probability, since it operates directly on the structure-agnostic two-groups model in~\eqref{eq:generic_two_groups_model}. The approaches in~\citet{ignatiadis2022confidence, xue2023nonparametric} can take the structural model in~\eqref{eq:structural} as the primitive, and so may yield shorter intervals.}
\end{itemize}

\section{Discussion}
\label{sec-discussion}

Our interpretation of the two two-groups models and local false discovery rate notions (local null-signal rate and complementary local activity rate) rely on the structural model~\eqref{eq:structural}. We briefly discuss the implication of these modeling assumptions and relate our results to the broader literature.

\paragraph*{On the structural model:} It is common in the literature on multiple testing and local false discovery rates, to only model the distribution of the test statistics $Y$ conditional on  $X=0$. The distribution of $Y \mid X=x$ for $x \neq 0$ may be left unspecified, which can be convenient for the data analyst. However, with such an approach, there is no obvious way to deduce anything about the signal $X$: if we aim to say something about the signal and its strength, we must start with a signal distribution and its relation with the response as in the structural model in~\eqref{eq:structural}.

There is scope for generalizing our results beyond the structural model in~\eqref{eq:structural}. For example, it is plausible 
that $\V(x) := \Var(Y \mid X=x) \neq 1$ for $x \neq 0$, see e.g.,~\citet{efron2010correlated}. Then, one could posit that $Y \mid X \sim \mathrm{N}(X, \V(X))$. If $\V(\cdot) > 0$ is such that $\V(0) = 1$, $\V(-x)=\V(x)$, then most of our results continue to be applicable with suitable modifications. For example, the probability of an active signal ($\A=1$) should be taken equal to $1-\exp\{-X^2/2 \V(X)\}$ in lieu of $1-\exp(-X^2/2)$ in~\eqref{eq:activity_given_x}. 

\paragraph*{On symmetry and two-sided p-values:} Our assumption of symmetry of the signal distribution $\P$ in~\eqref{eq:structural} is strong; we relaxed it in Section~\ref{sec:asymmetry-compatibility}. However, the predominant approach to testing in practice starts with the computation of the two-sided p-value \smash{$P_{\text{val}} :=  2\Phi(-|Y|)$}, where $\Phi$ is the standard normal distribution function. Indeed, in our computation of the estimated `local false discovery rates' with \texttt{qvalue} and \texttt{fdrtool} in the simulation of Section~\ref{sec:zero_assumptions_inactive}, we first transformed each $Y_i$ to a two-sided p-value.\footnote{\texttt{locfdr} and \texttt{tailinflation} operate directly on $Y_i$ and their estimates are not invariant to sign changes of the $Y_i$.} Under the structural model~\eqref{eq:structural}, $P_{\text{val}}$ has the following Lebesgue density conditional on $X$,
\begin{equation}
\begin{aligned}
\label{eq:pvalue_density}
f(p \mid x) = \frac{\phi(q_{p/2}+x) +\phi(q_{p/2}-x) }{2 \phi(q_{p/2})},\;\; 
\end{aligned}
\end{equation}
where $q_{p/2}$ is the $(1-p/2)$-quantile of the standard normal distribution and $p\in[0,1]$. For $x=0$,~\eqref{eq:pvalue_density} is just the density of the $\mathrm{Uniform}(0,1)$ distribution. All our results stated in terms of $Y$ may be stated in terms of the two-sided p-value $P_{\text{val}}$. Importantly, the conditional density of the p-value is symmetric in $x$, that is, $f(p \mid x) = f(p \mid -x)$. Thus, the implication of our results for $P_{\text{val}}$ do not hinge on the symmetry of the signal distribution $\P$ as we may symmetrize $\P$ without loss of generality.

\paragraph*{Further connections to related work:}  Several authors have argued in favor of replacing or complementing the local null-signal rate, $\lnsr(y)$~\eqref{eq:localnsr}, with $\P(|X| < \delta \mid Y=y)$ for a fixed (pre-specified) $\delta$~\citep*{ruppert2007exploring,vandewiel2007estimating,morris2008comment}. Our findings in the sparse-limit demonstrate that 
\begin{align*}
    \lidr(Y) \approx \P(|X| < \delta(Y) \given Y).
\end{align*}
While in both cases, it is clearly understood that zero signals and false discoveries are not synonymous, the definitions are not identical. The distinction lies in whether the threshold to categorize a signal as false is constant or varies based on the observation $Y$. \\

Our work naturally raises the following question: is there a direct probabilistic interpretation, analogous to our interpretation of the local false discovery rate, of Efron's~\citeyearpar{efron2004largescale} empirical null modeling strategy?\\

\subsubsection*{Reproducibility:} We provide code to reproduce Figures~\ref{fig:two_groups},~\ref{fig:simple_lfdr_simulation} and~\ref{fig:two_groups_asymmetric} on Github, under the repository:\\ 
\texttt{https://github.com/nignatiadis/local-\\
fdr-interpretation-paper}

\bibliographystyle{plainnat}
\bibliography{reference}

\begin{appendix}

\section{Appendix}
\label{appn}
\label{sec-proofs}

\begin{proof}[Continued proof of Theorem \ref{theo:optim}]
    Suppose for the sake of contradiction that there exists $\rho'<\rho$ and $\psi$ for which 
    \begin{align*}
        m(y)=(1-\rho')\phi(y)+\rho' \psi(y).
    \end{align*}
    Substituting $m(y)=(1-\pi_1)\phi(y)+\pi_1 m_1(y)$, we have
    \begin{align*}
        \psi(0) = \frac{(\rho'-\pi_1) \phi(0)  + \pi_1m_1(0)}{\rho'}.
    \end{align*}
    By assumption, $\rho' < \rho = \pi_1(1-m_1(0)/\phi(0))$. This implies
    \begin{align*}
        (\rho'-\pi_1)\phi(0)+\pi_1 m_1(0) < 0,
    \end{align*}
    a contradiction.
\end{proof}

\subsection{Proofs under the sparse limit approximation}
\begin{proof}[Proof of Theorem \ref{thm-constant-thresholds}]

    The conditional probability $\P(|XY| < \gamma \given Y)$ may be split into active and inactive components
\begin{align*}
&=
\frac{\phi(y)}{\psi_\rho(y)}\int_{|x| < \e} \cosh(xy) e^{-x^2\!/2} \P(\dd x), \\
   &=\frac{\phi(y)}{\psi_{\rho}(y)}\left(\int_{\Real} e^{-x^2\!/2} \P(\dd x)  - 2\rho I_0 + 2 \rho I_1\right) , \\
&= \frac{(1-\rho)\phi(y)}{\psi_{\rho}(y)} \cdot \frac{1-\rho - 2\rho I_0 + 2 \rho I_1}{1-\rho} ,
\end{align*}
where $I_0, I_1$ are defined in \eqref{def-i0-i1} and $\e=\gamma/|y|$. Since $e^{-x^2/2}\leq 1$,
\begin{align*}
    I_0 &= \rho^{-1} \int_{\e}^{\infty} e^{-x^2/2} \P(\dd x) \leq \rho^{-1} \P(|X|>\e).
\end{align*}
Since $\e \to 0$ as $\rho \to 0$, Markov's inequality implies 
\begin{align*}
    \P(|X|\wedge 1 \geq \e) 
    \leq \frac{\rho \int (x^2 \wedge 1) \H(\dd x)}{\e^2}+o(1),
\end{align*}
by \eqref{def:sparsity}, which implies $\rho I_0 \leq O(\rho/\e^2) = O(\rho |y|^2)$, which goes to zero. Next, notice that for some $C>0$, we have
\begin{align*}
    \cosh(xy)-1\leq C(xy)^2\text{ for all } x \in (0,\e)
\end{align*} 
because $\e y = \gamma$ a fixed constant; for instance $C = \cosh(\gamma)-1$ implies the above inequality. This implies
\begin{align*}
    I_1 &\sim \int_{0}^{\e} (\cosh(xy)-1) e^{-x^2/2} \H(\dd x) \leq O(y^2) 
\end{align*}
as $y \to \infty$. Thus, provided that $y = o(\rho^{-1/2})$, the limiting ratio in \eqref{eq:lfdr-equivalence} is one for all $\gamma$. 
    
\end{proof}

\begin{proposition}
\label{prop-threshold-form}
Suppose $\P(\dd x)$ tends to $\H(\dd x)$ as $\rho \to 0$, and $y = O(\sqrt{\log \rho^{-1}}) \to \infty$. It holds that 
$$
 \P\left( |X| < \tol(y) \mid Y=y \right) = \frac{(1-\rho)\phi(y)}{\psi_{\rho}(y)}(1+o(\rho)),
$$
for all  pairs $(H, \tol(\cdot))$ defined below:
\begin{enumerate}[leftmargin=15pt]
\item $\H(\dd x) = \de x/|x|^{1+\alpha}$ and $\tol(y)= \gamma_\alpha/|y|$, where $\gamma_\alpha$ is the solution to $J_\alpha(\gamma_\alpha) = 1/\alpha$.

\item $\H(\dd x) = |x|^{-1}e^{-\beta |x|} \de x$ and
\begin{align*}
    \tol(y) = \frac{1}{|y|}\log\left(\left(\log \frac{|y|}{\log\log |y|}\right)\cdot 2\log\log |y|\right).
\end{align*}
\item $\H(\dd x) = \dd x/(1+x^2)$ and $\tol(y)=\frac{1}{|y|}\log(2I_0|y|)$.
\item $\H(\dd x) = x^{\alpha-1} e^{-\beta x} \dd x$ for some $\alpha,\beta>0$ and
\begin{align*}
    \tol(y) =\frac{1}{|y|} \log \frac{c |y|^\alpha}{(\alpha\log |y|)^{\alpha-1}},
\end{align*}
where $c = c(\alpha, \beta) > 0$ is constant in $y$.
\end{enumerate}
\end{proposition}

\begin{proof}[Proof of Proposition \ref{prop-threshold-form}]

The conditional density of $|X|$ given $Y$ at $x>0$ is equal to
\begin{align*}
    &= \frac{\phi(y)}{\psi_{\rho}(y)} e^{xy-x^2/2}\P(\dd x) + \frac{\phi(y)}{\psi_{\rho}(y)} e^{-xy-x^2/2}\P(\dd x) \\ 
    &= \frac{2\phi(y)}{\psi_{\rho}(y)}e^{-x^2/2}\cosh(xy) \P(\dd x),
\end{align*}
which can equivalently be expressed as the mixture
\begin{align*}
\frac{2\phi(y)}{\psi_{\rho}(y)}\left(e^{-x^2/2}\P(\dd x) + (\cosh(xy)-1)e^{-x^2/2}\P(\dd x) \right).
\end{align*}
It follows that for any threshold $\tol(y)$, the conditional probability $\P(|X| \leq \tol(y) \mid Y=y)$ is
\begin{align*}
    &= \frac{2\phi(y)}{\psi_{\rho}(y)} \int_0^{\tol(y)} e^{-x^2/2} + (\cosh(xy)-1)e^{-x^2/2}\P(\dd x).
\end{align*}
From this we deduce that the condition \eqref{cond:exact-threshold} holds when $I_0 - I_1 \to 0$ as $\rho \to 0$, where
\begin{equation}
    \label{def-i0-i1}
    \begin{aligned}
    I_0 &:= \rho^{-1}\int_{\tol(y)}^\infty e^{-x^2/2} \P(\dd x),\\
    I_1 &:= \rho^{-1} \int_0^{\tol(y)} (\cosh(xy)-1) e^{-x^2/2}\P(\dd x).
\end{aligned}
\end{equation}
By \eqref{eq:exceedance_limit}, the integrals above have sparse-limit approximations
\begin{align*}
I_0 &= \int_{\tol(y)}^\infty e^{-x^2\!/2} \H(\dd x) + o(1),\\
I_1 
&=\int_0^{\tol(y)} \bigl(\cosh(xy) - 1) e^{-x^2\!/2} \H(\dd x)+o(1),
\end{align*}
as $\rho \to 0$. 
In what remains, the following function arises 
\[
J_\alpha(x) = \sum_{r=1}^\infty \frac{x^{2r}} {(2r)!\, (2r-\alpha)}
\]
for $\alpha < 2$. Note that  $J_\alpha(x)$ is increasing as a function of $\alpha$, the behaviour for large~$x$ in all cases is $e^{x}/(2x)$, and $J_{-1}(x) = \sinh(x)/x - 1$. For each of the four cases in the statement of Proposition \ref{prop-threshold-form}, we show $I_0-I_1 \to 0$. 
    \begin{enumerate}[leftmargin=15pt]
        \item We have
        \begin{align*}
            I_0 \sim \int_{\tol(y)}^\infty e^{-x^2/2} \frac{1}{x^{\alpha+1}} \de x \leq \int_{\tol(y)}^\infty x^{-\alpha-1} \de x \sim \frac{\tol^{-\alpha}}{\alpha}.
        \end{align*}
        The same expression holds as a lower bound,
        \begin{align*}
            I_0 &\sim \int_{\tol(y)}^\infty e^{-x^2/2} \frac{1}{x^{\alpha+1}} \de x \\
            &\geq \int_{\tol(y)}^1 (1-x^2/2) \cdot x^{-\alpha-1} \de x \sim \frac{\tol^{-\alpha}}{\alpha}, 
        \end{align*}
        since $\tol \to 0$. Next, we have
        \begin{align*}
            I_1 &\sim \int_0^{\tol(y)} (\cosh(xy)-1) e^{-x^2/2} \cdot x^{-\alpha-1} \de x \\
            &\sim \int_0^{\tol(y)} \sum_{n=1}^\infty \frac{(xy)^{2n}}{(2n)!} \cdot x^{-\alpha-1} \de x \\
            &= \sum_{n=1}^\infty \frac{y^{2n}}{(2n)!} \int_0^{\tol(y)} x^{2n-\alpha-1} \de x \\
            &= \tol^{-\alpha} \sum_{n=1}^\infty \frac{(\tol y)^{2n}}{(2n)!(2n-\alpha)} = \tol^{-\alpha} J_{\alpha}(\tol y).
        \end{align*}
        It follows that $I_1 \sim \tol^{-\alpha} J_\alpha(\gamma_\alpha) \sim \tol^{-\alpha} /\alpha \sim I_0$, so that $I_0 - I_1 \to 0$ as $\rho \to 0$.

\item We have
        \begin{align*}
            I_0 &\sim \int_{\tol(y)}^\infty e^{-x^2/2} x^{-1} e^{-\beta x} \de x \leq 
            - \log \tol(y).
        \end{align*}
        Also note that $\int_{\tol(y)}^\infty e^{-x^2/2-\beta x} x^{-1} \de x$ is 
        \begin{align*}
            &\geq \int_{\tol(y)}^1 (1-x^2/2-\beta x) x^{-1} \de x              \sim -\log \tol(y).
        \end{align*}
        Next, 
        \begin{align*}
            I_1
            &\sim \int_0^{\tol(y)} (\cosh(xy)-1) \cdot e^{-x^2/2-\beta x}x^{-1} \de x \\
            &\sim \sum_{n=1}^\infty \frac{y^{2n}}{(2n)!} \int_0^{\tol(y)} x^{2n-1} \de x \sim J_0(y\tol(y)).
        \end{align*}
        Now using $J_0(x) \sim \frac{e^x}{2x}$ for large $x$ and noting that $y\tol(y) \sim \log \log |y|$ as $|y|\to \infty$, we have
        \begin{align*}
            I_1 \sim \frac{e^{y\tol(y)}}{2y\tol(y)} \sim \log \frac{|y|}{\log\log |y|} 
            &\sim -\log\tol(y) \sim I_0,
        \end{align*}
        so that $I_0-I_1 \to 0$ as $\rho \to 0$.

\item As $\rho \to 0$, since $\tol(y) \to 0$, we have
        \begin{align*}
            I_0 &= \rho^{-1} \int_{\tol(y)}^{\infty} e^{-x^2/2} \P (\de x) \to \int_{0}^\infty  \frac{e^{-x^2/2}}{1+x^2}\de x
        \end{align*}
        which equals $0.822$. Since $\lim_{x\to 0}\frac{e^{-x^2/2}}{1+x^2}= 1$, 
        \begin{align*}
            I_1 
            &\sim \int_0^{\tol(y)}(\cosh(xy)-1) \cdot \frac{ e^{-x^2/2}}{1+x^2} \de x \\
            &\sim \int_0^{\tol(y)} \sum_{n=1}^\infty \frac{(xy)^{2n}}{(2n)!} \de x =\tol(y) \sum_{n=1}^\infty \frac{(y\tol (y))^{2n}}{(2n)!(2n+1)}, 
        \end{align*}
        equal to $\tol(y) J_{-1}(y\tol(y)) \sim 0.822$, since
        \begin{align*}
            \tol(y) J_{-1}(y\tol(y)) 
            &\sim \tol(y) e^{y\tol(y)}/(2y\tol(y)) 
            = I_0.
        \end{align*}
        
        \item Since $\tol(y) \to 0$ as $\rho \to 0$, we have
        \begin{align*}
            I_0 
            &= \rho^{-1} \int_{\tol(y)}^\infty \P(\dd x) - \rho^{-1} \int_{\tol(y)}^\infty (1-e^{-x^2/2})\P(\dd x) 
            \end{align*}
            \begin{align*}
            &\to c/2 := \Gamma(\alpha)/\beta^\alpha - \int_0^\infty (1-e^{-x^2/2})x^{\alpha-1}e^{-\beta x} \de x
        \end{align*}
        by the sparse limit approximation \eqref{def:sparsity}. 
        Next,
        \begin{align*}
            I_1 
            &\sim \int_0^{\tol(y)} (\cosh(xy)-1)e^{-x^2/2} x^{\alpha-1}e^{-\beta x} \de x \\
            &\sim \int_0^{\tol(y)} \sum_{n=1}^\infty \frac{y^{2n}}{(2n)!} x^{2n+\alpha-1} e^{-x^2/2-\beta x} \de x \\
            &= \tol(y)^{\alpha} \sum_{n=1}^\infty \frac{(y\tol(y))^{2n}}{(2n)!(2n+\alpha)} = \tol(y)^\alpha J_{-\alpha}(y\tol(y)).
        \end{align*}
        Now since $J_{-\alpha}(x) \sim \frac{e^x}{2x}$ as $x \to \infty$, we have
        \begin{align*}
            I_1 &\sim \tol(y)^\alpha e^{y\tol(y)}/(2y\tol(y)) \sim c/2,
        \end{align*}
        since $\log \frac{c y^\alpha}{ (\alpha\log y)^{\alpha-1}} \sim \alpha \log y$
        as $y \to \infty$. 
    \end{enumerate}
\end{proof}

\subsection{Proof that uniform consistency for local null-signal rate is impossible}
\label{sec-minimax-proof}
\begin{proof}[Proof of Proposition~\ref{prop:practically_unidentified}]
Take any $\P \in \mathcal{P}_{y,\eta}$ such that $\tilde{\pi}_0 := \P(X=0)>0$. Then consider the family of perturbed distributions, parameterized by $\xi > 0$,
$$ \P_{\xi} := \P + \tilde{\pi}_0(\delta_{\xi}/2 + \delta_{-\xi}/2 -  \delta_0).$$
For some $\bar{\xi} > 0$, it will hold that $\P_{\xi} \in \mathcal{P}_{y,\eta}$ for all $0<\xi \leq \bar{\xi}$. The total variation distance of the marginals satisfies
$$
\begin{aligned}
&\TV( \mathrm{N}(0,1)* \P,\, \mathrm{N}(0,1)* \P_{\xi}) \\
&= \frac{1}{2}\int \left|  \int \phi(y-x)\{d\P(x)-d\P_{\xi}(x)\} \right | dy = O(\xi).
\end{aligned}
$$
Next note that $\lnsr_\P(y) > 0$, while $\lnsr_{\P_{\xi}}(y) =0$, that is, 
$$\lnsr_\P(y) - \lnsr_{\P_{\xi}}(y) = \lnsr_\P(y) > 0,$$ 
for all $\xi >0$.
The conclusion follows by Le Cam's two-point method upon taking $\xi = \xi_n \searrow 0$ sufficiently fast, see e.g.,~\citet[Theorem 2.2]{tsybakov2008introduction}.
\end{proof}

\subsection{Estimation of complementary local activity rate}
\label{subsec:estimating_clar}

\begin{proposition}
\label{prop:clar_easy_to_estimate}
Take any sequence $\varepsilon_n \to 0$ with $n \varepsilon_n^2 / \surd{\log n} \to \infty$. Then there exists an estimator of $\clar(y)$, $\widehat{\lidr}_n(y)= \widehat{\lidr}_n(y; Y_1,\dotsc,Y_n)$, based on independent observations $Y_1,\dotsc,Y_n$ from model~\eqref{eq:structural} such that:
$$\limsup_{n \to \infty} \sup_{\P \in \mathcal{P}_{y,\eta}} \P( |\widehat{\lidr}_n(y) - \lidr(y)| > \varepsilon_n) = 0.$$
\end{proposition}
The uniform rate derived is barely slower than the parametric rate $1/\surd{n}$. Such a fast uniform estimation rate also implies that it is possible to construct short confidence intervals for $\lidr(y)$~\citep{ignatiadis2022confidence}. 

\begin{proof}
Consider the estimator
$$ \widehat{\lidr}(y) := \frac{ \phi(y) \sum_{i=1}^n K(Y_i/h_n)}{ \phi(0) \sum_{i=1}^n K\{(Y_i - y)/h_n\}},$$
where $K(u) = \sin(u)/(\pi u)$ is the sinc kernel and the bandwidth is chosen as $h_n = 1/\surd{\log n}$. The proof of the claimed uniform rate follows as in~\citet[ Propositions S9 and S10 in the supplement]{ignatiadis2022confidence}.

\end{proof}

\subsection{An extension for compatible signal distributions}
\label{subsec:compatible-extension}

In this section, we construct a joint distribution for $(X,Y,\A)$ that is consistent with the bivariate marginalizations for $(X,Y)$ and $(Y,\A)$. Our construction begins with the weighted hyperbolic secant function, defined as
\begin{align}
\label{eq:weighted-sech}
    \sech_w(y) := \left[ \int e^{xy}w(\dd x) \right]^{-1},
\end{align}
where $w(\dd x) = e^{-x^2/2}\P(\dd x)/\rho_0$ is a re-weighted version of the signal distribution $\P$, and $\rho_0 = \int e^{-x^2/2}\P(\dd x)$. The joint density for $(X,Y,\A)$ is defined
\begin{align*}
    \begin{cases}
        \phi(y) \; e^{xy-x^2/2} \P(\dd x)\sech_w(y) \; &\text{ if } \A=0 \\
        \phi(y) \; e^{xy-x^2/2} \P(\dd x)(1-\sech_w(y)) &\text{ if } \A=1.
    \end{cases}
\end{align*}
The above joint distribution implies $\mathcal{A} \indep X \mid Y$, i.e.
\begin{align}
\label{eq:activity-joint-conditional-prob}
    \P(\A=0 \mid X=x,Y=y) = \sech_w(y).
\end{align}
\begin{proposition}
    If $\P$ is compatible with the zero assumption, then \eqref{eq:activity-joint-conditional-prob} is a valid probability and the joint distribution for $(X,Y,\A)$ displayed above is consistent with the bivariate distributions for $(X,Y)$ and $(\A,Y)$:
    \begin{align*}
        X &\sim \P(\dd x), \hspace{1em} Y \mid X \sim \mathrm{N}(X,1) \\
        \A &\sim \textnormal{Bernoulli}(1-\rho_0), \hspace{1em} Y \mid \A \sim \begin{cases}
            \phi \hspace{1em} &\text{if }\A=0\\
            \psi_1 &\text{if }\A=1,
        \end{cases}
    \end{align*}
    where $\psi_1(y) := \rho_1^{-1}\phi(y) \int (e^{xy}-xy-1)e^{-x^2/2}\P(\dd x)$.
\end{proposition}
\begin{proof}
We have
\begin{align*}
    \cosh_w(y) &:= 1/\sech_w(y) \\
    &= \rho_0^{-1} \int e^{xy-x^2/2}\P(\dd x) \\
    &= 1+\rho_0^{-1}  \int (e^{xy}-1)e^{-x^2/2}\P(\dd x) \\
    &= 1+\rho_0^{-1}  \int (e^{xy}-xy-1)e^{-x^2/2}\P(\dd x),
\end{align*}
where the last line follows from compatibility (Theorem \ref{thm-compatibility}). Since the integrand is positive, $\cosh_w(y) \geq 1$ which implies $\sech_w(y) \in [0,1]$, so \eqref{eq:activity-joint-conditional-prob} is a valid probability for any $y$. 

For the second claim, we have
\begin{align*}
    \P(x,y,0) + \P(x,y,1) = \int \phi(y-x)\P(\dd x),
\end{align*}
so the joint density is consistent with the distribution for $(X,Y)$. Finally, integrating out $x$ in the $\A=0$ case, we have
\begin{align*}
    \int \P(\dd x,y,0) &= \phi(y) \sech_w(y) \int e^{xy-x^2/2}\P(\dd x)  \\
    &= \rho_0 \phi(y).
\end{align*}
Integrating out $x$ in the $\A=1$ case gives
\begin{align*}
    \int \P(\dd x,y,1) &= \phi(y)(1-\sech_w(y)) \int e^{xy-x^2/2} \P(\dd x) \\
    &= \phi(y) \; \frac{\int e^{xy}w(\dd x)-1}{\int e^{xy}w(\dd x)}\int e^{xy-x^2/2} \P(\dd x) \\
    &= \phi(y) \cdot \rho_0 \int (e^{xy} -1)w(\dd x)  \\
    &= \phi(y) \int (e^{xy}-1) e^{-x^2/2}\P(\dd x) \\
    &= \phi(y) \int (e^{xy}-xy-1) e^{-x^2/2}\P(\dd x),
\end{align*}
where the last line follows from compatibility. Dividing and multiplying by $\rho_1$ gives the desired result,
\begin{align*}
    \int \P(\dd x,y,1) = \rho_1 \psi_1(y).
\end{align*}
\end{proof}

As the previous result shows, the joint distribution of $(X,Y,\A)$ defined via \eqref{eq:activity-joint-conditional-prob} is consistent with $(X,Y)$ and $(\A,Y)$. However, it is inconsistent with the specification \eqref{eq:activity_given_x} of $(X,\A)$ given in Section \ref{sec-probabilistic-interpretation}; marginalization of $y$ in \eqref{eq:activity-joint-conditional-prob} yields
\begin{align*}
    \P(\A=0 \mid X=x) = e^{-x^2/2} \int \phi(y) e^{xy} \sech_w (y) \dd y.
\end{align*}
Moreover, an implication of the definition \eqref{eq:activity-joint-conditional-prob} is that $X \indep \A \mid Y$, meaning that the event $\A = 0$ does not imply anything about the signal $X$ given $Y$, e.g. that it is small in some sense.

\end{appendix}

\end{document}